\documentclass{amsart}
\usepackage{amssymb,amsrefs}
\usepackage[colorlinks]{hyperref}
\usepackage{booktabs}
\usepackage[all]{xy}

\theoremstyle{plain}
\newtheorem{theorem}{Theorem}[section]
\newtheorem{proposition}[theorem]{Proposition}
\newtheorem{lemma}[theorem]{Lemma}
\newtheorem{corollary}[theorem]{Corollary}

\theoremstyle{definition}
\newtheorem{remark}[theorem]{Remark}
\newtheorem{definition}[theorem]{Definition}

\newtheorem{algorithm}[theorem]{Algorithm}

\DeclareMathOperator{\characteristic}{char}
\DeclareMathOperator{\Cl}{Cl}
\DeclareMathOperator{\End}{End}
\DeclareMathOperator{\GL}{GL}
\DeclareMathOperator{\Jac}{Jac}
\DeclareMathOperator{\Norm}{Norm}
\DeclareMathOperator{\ord}{ord}

\DeclareMathOperator{\col}{:}

\newcommand{\CC}{{\mathbf{C}}}
\newcommand{\FF}{{\mathbf{F}}}
\newcommand{\PP}{{\mathbf{P}}}
\newcommand{\QQ}{{\mathbf{Q}}}
\newcommand{\ZZ}{{\mathbf{Z}}}

\newcommand{\gotha}{{\mathfrak{a}}}
\newcommand{\gothabar}{\bar{\gotha}}                                            
\newcommand{\gothp}{{\mathfrak{p}}}
\newcommand{\gothpbar}{\bar{\gothp}}

\newcommand{\calC}{{\mathcal{C}}}
\newcommand{\calH}{{\mathcal{H}}}
\newcommand{\calO}{{\mathcal{O}}}

\newcommand{\Kp}{K^+}         
\newcommand{\eps}{\varepsilon}

\renewcommand{\hat}{\widehat}
\renewcommand{\tilde}{\widetilde}
\renewcommand{\mod}{\bmod}

\newcommand\lowtilde{\lower0.7ex\hbox{\textasciitilde}}

\newcommand{\mybar}[1]{
  \mathchoice
  {#1\llap{$\overline{\phantom{\displaystyle\rm#1}}$}}
  {#1\llap{$\overline{\phantom{\textstyle\rm#1}}$}}
  {#1\llap{$\overline{\phantom{\scriptstyle\rm#1}}$}}
  {#1\llap{$\overline{\phantom{\scriptscriptstyle\rm#1}}$}}
}  
\renewcommand{\bar}{\mybar}

\newcommand{\kb}{\bar{k}}
\newcommand{\nubar}{\bar{\nu}}
\newcommand{\pibar}{\bar{\pi}}

\newlength{\algindent}
\newlength{\alglabel}
\newcounter{stepcount}

\newenvironment{algtop}
{\quad\begin{list}{\arabic{stepcount}.}{\leftmargin=\algindent\labelwidth=\algindent\itemsep=\smallskipamount\usecounter{stepcount}}}
{\end{list}}

\newcommand{\algin}{\item[\emph{Input}:]}
\newcommand{\algout}{\item[\emph{Output}:]}

\newenvironment{alglist}
{\quad\begin{list}{\arabic{stepcount}.}{\leftmargin=1.5em\labelwidth=1em\labelsep0.5em\itemsep=\smallskipamount\usecounter{stepcount}}}
{\end{list}}

\newcounter{substepcount}
\newenvironment{algsublist}
{\quad\begin{list}{(\/\rlap{\alph{substepcount}}\phantom{d}\/)}{\usecounter{substepcount}}}
{\end{list}}

\begin{document}

\title[Genus-2 curves and Jacobians]
{Genus-2 curves and Jacobians\\ with a given number of points}

\author[Br\"oker]{Reinier Br\"oker}        
\address{Department of Mathematics,
         Brown University,
         Box 1917,
         151 Thayer Street,
         Providence, RI 02912, USA}
\email{reinier@math.brown.edu}

\author[Howe]{\hbox{Everett W. Howe}}      
\address{Center for Communications Research,
         4320 Westerra Court,
         San Diego, CA 92121, USA}
\email{however@alumni.caltech.edu}
\urladdr{http://www.alumni.caltech.edu/\lowtilde{}however/}

\author[Lauter]{\hbox{Kristin E. Lauter}}   
\address{Microsoft Research,
         One Microsoft Way,
         Redmond, WA 98052, USA}
\email{klauter@microsoft.com}

\author[Stevenhagen]{\hbox{Peter Stevenhagen}}
\address{Mathematisch Instituut,
         Universiteit Leiden, 
         Postbus 9512, 
         2300 RA Leiden, The Netherlands}
\email{psh@math.leidenuniv.nl}

\date{28 March 2014; revised 10 August 2014}


\begin{abstract}
We study the problem of efficiently constructing a curve $C$ of genus~$2$ over
a finite field $\mathbf{F}$ for which either the curve $C$ itself or its 
Jacobian has a prescribed number $N$ of $\mathbf{F}$-rational points.

In the case of the Jacobian, we show that any `CM-construction' to produce the
required genus-$2$ curves necessarily takes time exponential in the size of its
input.

On the other hand, we provide an algorithm for producing a genus-$2$ curve with
a given number of points that, heuristically, takes polynomial time for most 
input values.  We illustrate the practical applicability of this algorithm by
constructing a genus-$2$ curve having exactly $10^{2014} + 9703$ (prime) 
points, and two genus-$2$ curves each having exactly $10^{2013}$ points.

In an appendix we provide a complete parametrization, over an arbitrary base 
field $k$ of characteristic neither $2$ nor~$3$, of the family of genus-$2$ 
curves over $k$ that have $k$-rational degree-$3$ maps to elliptic curves, 
including formulas for the genus-$2$ curves, the associated elliptic curves,
and the degree-$3$ maps.
\end{abstract}

\maketitle

\section{Introduction}
\label{S:intro}

\noindent
For an algebraic variety $V$ defined over a finite field $\FF$, a fundamental
quantity is its number $N=\# V(\FF)$ of $\FF$-rational points.  This quantity,
which we briefly refer to as the \emph{order} of $V$ over $\FF$, can be found
by a finite computation. However, in view of computer calculations and 
cryptographic applications, the problem of \emph{efficiently} counting the 
number of $\FF$-rational points of smooth algebraic varieties defined over a 
finite field has become a topic of intensive research in the last $25$ years.

There is a natural inverse to the point counting problem, which is 
mathematically and cryptographically interesting as well.  It is the problem of
efficiently constructing, for a given integer $N$, a smooth variety $V$ over a 
finite field $\FF$ such that $V$ has order $N$ over $\FF$.  Usually, one 
restricts the class of varieties $V$ under consideration by requiring that $V$
be, for example, a curve of a given genus or a surface of a given type.  In all
cases, the question can be phrased in two different ways with respect to~$\FF$.
One may either 
\begin{description}
\item [A] take \emph{both} $N$ and $\FF$ as input, and construct a variety $V$
      of order $N$ over $\FF$, or
\item [B] take only $N$ as input, and construct $\FF$ \emph{and} a variety $V$
      of order $N$ over $\FF$.
\end{description}
In the case of curves of genus~$1$, that is, \emph{elliptic curves}, it is a
major open problem to find an efficient algorithm for Problem A.  The main 
result of~\cite{BS} is that, at least heuristically, Problem B for elliptic 
curves admits an efficient solution if the input~$N$ is provided to the
algorithm in factored form.  In this paper, we generalize this result to curves
of genus~$2$.  More precisely, Problem B for elliptic curves admits \emph{two} 
natural analogues in higher genus, and our Main Theorems \ref{Th:J} 
and~\ref{Th:C} show that they give rise to rather different answers.

For a smooth, projective curve $C$ of genus $g$ defined over a finite 
field~$\FF$, the Jacobian $J=\Jac C$ of $C$ is a $g$-dimensional abelian
variety over $\FF$, and $J(\FF)$ is a finite abelian \emph{group}.  Using a
base point in $C(\FF)$, one may embed $C$ into $J$ under the Abel--Jacobi map.
In the elliptic case $g=1$, this leads to an identification of $C$ and~$J$,
but in higher genus, $C$ is a strict subvariety of $J$, and $C(\FF)$ is merely
a subset of the group $J(\FF)$. This leads to two mathematically natural
generalizations in genus $2$ of the construction problem for elliptic curves:
\begin{enumerate}
\item construct \emph{curves} of genus 2 of given order, or
\item construct curves of genus 2 with \emph{Jacobians} of given order.
\end{enumerate}
From a cryptographic point of view, the second generalization is the relevant
one, as current applications use the group $J(\FF)$ rather than the
set~$C(\FF)$.  We will consider both generalizations, taken in the setting of
Problem~B.

In Section \ref{S:Jacobians}, we consider the problem of efficiently
constructing, for a given integer $N$, a finite field $\FF$ and a genus-$2$
curve $C$ defined over $\FF$ such that $J=\Jac C$ has order $N$ over $\FF$.
In this case, we say that the pair $(C, \FF)$ \emph{realizes} $N$.  For our
purposes, it suffices to consider only \emph{quartic pairs} $(C, \FF)$ 
realizing $N$.  These are pairs for which the subring $\ZZ[\pi]\subset \End J$
generated by the Frobenius element $\pi$ in the endomorphism ring of $J$ is an
order in a quartic CM-field $K=\QQ(\pi)$; this condition is equivalent to the
characteristic polynomial of Frobenius for $C$ being irreducible.  The 
justification for restricting to quartic pairs is that non-quartic pairs will
only realize a zero-density subset of all possible input values~$N$; see 
Corollary~\ref{cor:split}.

In view of \cite{BS}, the natural approach to constructing quartic pairs 
$(C, \FF)$ realizing~$N$ consists in obtaining $C$ as the reduction of a 
genus-$2$ curve $\tilde C$ in characteristic zero with CM by $K$, since the 
Igusa modular invariants of such $\tilde C$ may be computed by CM-techniques.
As we explain at the end of Section \ref{S:Jacobians}, a so-called
\emph{CM-construction} of $(C, \FF)$ which, in an intermediary step, writes 
down the Igusa class polynomials of a curve $\tilde C$ in characteristic zero
with CM by $K$ that reduces to $C$ over $\FF$, is necessarily exponential in 
$\log\Delta_K$, the size of the discriminant $\Delta_K$ of $K=\QQ(\pi)$; see 
Corollary~\ref{cor:runtime}. As a consequence, CM-constructions are only 
computationally feasible for CM-fields~$K$ of small discriminant.

Given $N$, there are, up to isomorphism, only finitely many pairs $(C, \FF)$ 
realizing~$N$, so we may define the 
\emph{minimal genus-$2$ Jacobian discriminant} $\Delta(N)$ of~$N$ as the
smallest discriminant of a CM-field $K=\QQ(\pi)$ associated to a quartic pair
$(C, \FF)$ realizing $N$. There are two sets of $N$'s for which this definition
is not appropriate. The first is the zero-density set of those $N$ that are 
realized by non-quartic pairs.  The second is the set of those $N$ that are not
realized at all as orders of genus-$2$ Jacobians over finite fields. This is 
also a zero-density set (Theorem~\ref{thm:zerodensity}), and conjecturally it 
is even empty.  For $N$ in one of these two very thin sets, we formally put
$\Delta(N)=0$.  As we are to prove that $\Delta(N)$ tends to be large, this 
choice only strengthens Theorem~\ref{Th:J} below.

In the elliptic case~\cite{BS}, the expected minimal discriminant of the 
endomorphism algebra of an elliptic curve of order $N$ over a prime field 
grows, at least heuristically, as $(\log N)^2$, and this gives rise to 
efficient CM-constructions. For genus~$2$, we prove that this is not the case.

\begin{theorem}
\label{Th:J}
For an integer $N\in\ZZ_{>0}$, let $\Delta(N)$ be the minimal genus-$2$ 
Jacobian discriminant defined above.  Then we have
\[
\limsup_{N\to\infty} \frac{\Delta(N)}{\sqrt N}=+\infty.
\]
\end{theorem}

\noindent
This theorem implies that \emph{any} genus-$2$ CM-construction for abelian
surfaces over finite fields of prescribed order $N$, with `CM-construction'
taken in the sense explained above, will have a worst case run time that is
exponential in the input size $\log N$.  Section \ref{S:Jacobians}, which 
contains the proof of Theorem \ref{Th:J}, provides an explicit worst case 
lower bound for the run time of genus-$2$ CM-constructions solving our problem 
(Corollary~\ref{cor:runtime}), as well as a strengthening of the theorem 
(Corollary~\ref{cor:generic}) that shows that the growth of the lim sup does 
not come from a thin set of $N$'s, and that $\Delta(N)/\sqrt N$ becomes in fact
large for `most'~$N$.

The proof of Theorem \ref{Th:J} is based on a `scarcity' of small quartic
CM-fields~$K$ that contain the Weil numbers corresponding to pairs $(C, \FF)$ 
realizing $N$. The difference with the genus-$1$ situation lies in the fact
that the cardinality of the base field~$\FF$ for genus-$2$ Jacobians of order
$N$ is about $\sqrt N$, and not~$N$. If one requires the curve $C$ itself to
have order $N$ over $\FF$, this problem disappears, and one can hope that, just
as in the elliptic case, efficient constructions can be given.

The current state of our knowledge on gaps between prime numbers does not allow
us to \emph{prove} that an elliptic curve or genus-$2$ curve of order $N$ over
a finite field exists for all~$N$, but, heuristically and in computational 
practice, this is never a problem. In Section~\ref{S:curves}, we provide an
algorithm that efficiently finds genus-$2$ curves of order 
$N\not\equiv 1\mod 6$ in the following sense.

\begin{theorem}
\label{Th:C}
There exists an algorithm that, on input of an integer $N\not\equiv 1\mod 6$
together with its factorization, tries to return a prime number $p$ and a 
genus-$2$ curve $C/\FF_p$ of order $\#C(\FF_p)=N$.  If there exists an ordinary
elliptic curve of order $N$ over a prime field $\FF_p$ such that 
$p\equiv N-1\bmod\ell$ for $\ell = 2$ or $\ell = 3$, then the algorithm will be
successful. 

Under standard heuristic assumptions, the required elliptic curve exists for 
all $N\in\ZZ_{>1}$, and the expected run time of the algorithm is polynomial in
$2^{\omega(N)}\log N$.  Here $\omega(N)$ denotes the number of distinct prime
factors of $N$.
\end{theorem}

\noindent
Although the run time in Theorem \ref{Th:C} is not polynomial in the usual
sense, it is polynomial in $\log N$ outside a zero-density subset of
$\ZZ_{\geq 1}$ consisting of very smooth input values~$N$.

The hypothesis on the existence of an \emph{elliptic} curve of order $N$ in 
Theorem~\ref{Th:C} is caused by the fact that we construct the curve $C$ in the
theorem as a genus-$2$ curve with split Jacobian $J\sim E_1\times E_2$, and
this requires the construction of auxiliary elliptic curves $E_1$ and~$E_2$ of 
given orders.  Such elliptic curves can be constructed by the method 
of~\cite{BS} discussed in Section \ref{S:genus1}.  The Jacobian $J$ of $C$ is 
then obtained by \emph{gluing} $E_1$ and~$E_2$ along their $n$-torsion for some
integer $n>1$.  In this case, the genus-2 curve $C$ has the special property 
that it allows nonconstant maps to the elliptic curves $E_1$ and~$E_2$.  For
$n\le4$ this is a classical topic, at least when performed over the complex 
numbers.  It was already used in the $19$th century by Jacobi~\cite{Jacobi},
Goursat~\cite{Goursat}, and others to express hyperelliptic integrals in terms
of elliptic integrals.

In Section \ref{S:gluing} we give an algorithmic description of the gluing
results for $n=2$ and $n=3$ that keeps track of fields of definition.  For
$n=3$ we actually do a bit more, in an appendix to the paper:  We provide a 
complete parameterization of the genus-$2$ curves that admit rational 
degree-$3$ maps to an elliptic curve, including formulas for the genus-$2$
curves, the associated elliptic curves, and the degree-$3$ maps.  The explicit
gluing algorithms are used in the proof of Theorem~\ref{Th:C} in 
Section~\ref{S:curves}.  One result of the restriction to $n\in\{2, 3\}$ is the
congruence condition $N\not\equiv 1\mod 6$ in Theorem~\ref{Th:C}.  To handle 
arbitrary $N$ by our method, one would need to use explicit algorithms for 
gluing two elliptic curves along their $\ell$-torsion for arbitrary
primes~$\ell$, because only primes $\ell$ coprime to $N-1$ can be used.

In the final Section \ref{S:examples}, we illustrate the explicit working of 
our algorithm and construct two genus-2 curves of smooth order $10^{2013}$ and 
one of prime order $10^{2014}+9703$.  The construction of the prime order curve
required finding a root, in a large finite field, of a class polynomial for an 
imaginary quadratic order of large discriminant. We thank Andrew Sutherland for
his generous help in carrying out this calculation for us, using the methods
of~\cite{Sutherland}.

In this paper, we view all varieties as being schemes over a given base field.
It follows that morphisms of varieties are morphisms 
\emph{over that base field}.  For example, what we call the 
\emph{endomorphism ring} of an abelian surface over a field~$k$, other authors 
might call the ring of \emph{$k$-rational} endomorphisms.

\section{Elliptic curves of given order}
\label{S:genus1}

\noindent
We start with a review of the elliptic case.  Even though Theorem~\ref{Th:J}
expresses the fact that this case is rather different from the genus-$2$ case,
the elliptic case is used in an essential way in Section~\ref{S:curves}, in the 
proof of Theorem~\ref{Th:C}.

For an elliptic curve $E$ defined over a finite field $\FF_q$ of $q$ elements,
the order $N=\#E(\FF_q)$ is an integer in the Hasse interval 
\begin{equation}
\label{eq:Hasseqg=1}
\calH_q=[(\sqrt q-1)^2, (\sqrt q +1)^2] = [q+1-2\sqrt q,\  q+1+2\sqrt q]
\end{equation}
of length $4\sqrt q$ centered around $q+1$. Note that $N$ and $q$ are of the 
same size, and that we have a symmetric relation
\begin{equation}
\label{eq:Nqsymmgenus1}
N\in\calH_q \quad \Longleftrightarrow\quad q\in\calH_N.
\end{equation}
The integers contained in the union of the intervals $\calH_q$ for those 
fields $\FF_q$ that are \emph{not} prime fields form a zero-density subset of
$\ZZ_{>0}$, so any algorithm realizing elliptic curves of arbitrary prescribed 
order $N$ can safely restrict to prime fields $\FF_p$.  It is well-known that 
every integer $N\in\calH_p$ is realized by an elliptic curve over $\FF_p$, but 
unfortunately it is unproved that the union 
$\bigcup_{\textup{$p$ prime}} \calH_p$ of the Hasse intervals for prime fields 
contains all positive integers.  The problem here is that it is unknown whether
we have an upper bound
\begin{equation}
\label{eq:gapbound}
d_n = p_{n+1}- p_n < 4 \sqrt {p_n}
\end{equation}
for the prime gap $d_n$ following the $n$-th prime~$p_n$. This is the bound
that makes $\calH_{p_n}$ and $\calH_{p_{n+1}}$ overlap, and that would prevent
integers from being outside the Hasse intervals $\calH_p$ for all primes~$p$.
Currently, the best proven upper bound \cite{BHP} is $d_n< p_n^{.525}$, which 
is not good enough for our purposes.

It is possible to prove that only a very thin set of integers $N$ lies outside
all Hasse intervals. By a result of Matom{\"a}ki (see Lemma~\ref{L:Matomaki}),
the total length $\sum_n d_n$ of prime gaps $d_n > \sqrt {p_n}$ for the primes 
$p_n<X$ is no more than $O(X^{2/3})$, and this yields an upper bound for the 
number of integers up to $X$ that are \emph{not} the order of the group of 
points of an elliptic curve over a finite field.

Even though \eqref{eq:gapbound} is unproved, we know by the prime number 
theorem that, \emph{on average}, we have $d_n\approx \log p_n$. This means that
finding prime fields over which $N$ can be realized as the order of an elliptic
curve is never a practical problem.  As the expected number of possible $p$ for
a large value~$N$ is expected to be about $4\sqrt N/\log N$, there is ample 
choice in practice.

The key problem arising in the elliptic case is that, given an integer 
$N\in\calH_p$, the best general algorithm we know to construct an elliptic
curve over $\FF_p$ of order $N$ is the rather na\"{\i}ve method of picking
random elliptic curves over $\FF_p$ and checking whether their order 
equals~$N$, until a curve of order $N$ is found.  As checking the order (and 
even complete point counting) for elliptic curves over $\FF_p$ can be done in 
time polynomial in $\log p\approx \log N$, the run time for this na\"{\i}ve
probabilistic algorithm is essentially determined by the number of elliptic 
curves one has to try before one of order $N$ is encountered.  This expected
number is of order $\sqrt N$, and the resulting run time 
$O(N^{\frac{1}{2}+\eps})$ for any $\eps>0$ is exponential in $\log N$, the size
of the input value~$N$. This means that we do not obtain an efficient algorithm
to solve Problem~A from the Introduction in the case of elliptic curves.

The solution provided in \cite{BS} to construct elliptic curves of prescribed 
order over a given finite field uses a deterministic complex multiplication
approach, which has an even slower run time $O(N^{1+\eps})$ for \emph{most} 
values of $N$ and $p$.  However, it runs in polynomial time in cases where the 
curve to be constructed has a `small' endomorphism ring.  Heuristically,
suitable small endomorphism rings can always be found in the less restrictive 
setting of Problem~B, where, on input $N$, one is free to choose a prime field
$\FF=\FF_p$.

To make this more precise, we recall that for an elliptic curve $E$ 
over~$\FF_p$, the Frobenius endomorphism $\Phi_p \in\End E$ satisfies a 
quadratic equation
\[
\Phi_p^2 -t \Phi_p + p = 0\in\End E
\]
of discriminant $\Delta=t^2-4p<0$.  The associated \emph{Weil $p$-polynomial} 
\begin{equation}
\label{eq:Weilpolg=1}
f=T^2-tT+p\in\ZZ[T],
\end{equation}
which may be viewed as the characteristic polynomial of $\Phi_p$ acting on the
Tate module $T_\ell(E)$ of $E$ at a prime $\ell\ne p$, characterizes the
isogeny class of $E$, and the elliptic curves in this isogeny class are those
elliptic curves over $\FF_p$ that have order $f(1)=p+1-t$.

To construct, for a given Weil $p$-polynomial $f=T^2-tT+p$ of 
discriminant~$\Delta$, an elliptic curve $E$ in the corresponding isogeny 
class, one can use the \emph{complex multiplication method}, which realizes 
$E$ as the reduction modulo $p$ of an elliptic curve in characteristic zero.
More precisely, there are only finitely many isomorphism classes of 
\emph{complex} elliptic curves with endomorphism ring isomorphic to the 
imaginary quadratic order $\calO_\Delta=\ZZ[T]/(f)$.  Complex analytically,
these curves arise as quotients $\CC/I$ for invertible ideals
$I\subseteq \calO_\Delta\subset\CC$.  Their $j$-invariants depend only on the
class of the ideal $I$ in the class group $\Cl \calO_\Delta$, and they are 
\emph{algebraic integers} that form the zeroes of the Hilbert class polynomial
\begin{equation}
\label{eq:Hilbpol}
P_\Delta=\prod_{[I]\in\Cl \calO_\Delta} (T-j(I))\in\ZZ[T].
\end{equation}
The polynomial $P_\Delta$ splits into distinct linear factors in $\FF_p[T]$, 
and its roots in $\FF_p$ are the $j$-invariants of the elliptic curves over
$\FF_p$ that have Weil polynomial $f$.  It is trivial to write down an explicit 
model $E/\FF_p$, say in Weierstrass form, with given $j$-invariant 
$j(E)\in\FF_p$. The $j$-invariant $j(E)$ determines $E$ up to twists over
$\FF_p$ and, for $\Delta<-4$, the Weil polynomial $f=T^2\pm tT+p$ of $E$ up to 
the sign of~$t$. As it is easy to check which of the twists of $E$ has the 
desired order~$N$, finding $j(E)\in \FF_p$ solves our problem.

The polynomial $P_\Delta$ can be used to write down an elliptic curve over 
$\FF_p$ of order~$N$ if there exist elements $\nu$ and $\pi$ in the imaginary 
quadratic order $\calO_\Delta$ satisfying
\begin{equation}
\label{EQ:nupi}
\nu\nubar=N,\qquad
\nu+\pi=1,\qquad
\pi\pibar=p \textup{ (prime)}.
\end{equation}
Note that, despite the symmetry in $N$ and $p$, this is just a way to phrase
the fact that $N$ is the norm $\Norm(1-\pi)$ for a Weil $p$-number 
$\pi\in \calO_\Delta$.  As the degree of~$P_\Delta$, and heuristically also the
size of the coefficients of~$P_\Delta$, are of order of 
magnitude~$|\Delta|^{1/2}$, the time needed to compute $P_\Delta$ is
exponential in $\log|\Delta|$.  One therefore looks for the \emph{minimal}
imaginary quadratic order $\calO_\Delta$ in which there exist elements $\nu$ 
and~$\pi$ satisfying Equation~\eqref{EQ:nupi}. This order can in principle be
found by factoring $N$ in all possible ways as $N=\nu\nubar$ in $\calO_\Delta$
for ascending values of $|\Delta|$, until an element $\nu$ is found for which 
$\pi=1-\nu$ has prime norm $p$. It is explained in \cite{BS} how this can be
done efficiently in case the prime factorization of $N$ in $\ZZ$ is known and 
why, on input $N$, the expected minimal value of $|\Delta|$ for which $\pi$ is
found is heuristically of size $O((\log N)^2+2^{\omega(N)})$.  Here $\omega(N)$
denotes the number of distinct prime factors of~$N$. For our purposes, it 
suffices to know that the CM-construction of elliptic curves we sketched yields
the following result.

\begin{lemma}
\label{L:BS}
The CM-construction produces an elliptic curve $E$ over a prime field $\FF_p$
that solves Problem~B from the Introduction for factored input values $N$ in a
time that is heuristically polynomial in $2^{\omega(N)}\log N$.

For a fixed prime $\ell$, the same holds true under the additional restriction
that the prime $p$ be congruent to $N-1$ modulo $\ell$ and that the elliptic
curve $E$ be ordinary, provided that $N-1$ is not divisible by $\ell$.
\end{lemma}

\begin{proof}
This first statement is \cite{BS}*{Corollary~4.4}.  The proof given there also
explicitly formulates the heuristic assumption underlying the analysis in the
following way:  The elements $\nu$ behaving like \emph{random} quadratic
integers of norm $N$, in the sense that the norm $N+1-(\nu+\nubar)$ of 
$\pi=1-\nu$, which is an integer of the same order of magnitude as~$N$, will be
prime with frequency $1/\log N$. This random behavior of~$\nu$ will also be
reflected in the trace $t=\nu+\nubar$ of $\nu$ taken modulo~$\ell$, provided 
that we keep in mind that the residue class $(N-1+t\mod \ell)$ in which we find
our prime $p=N-1+t$ has to be invertible modulo $\ell$.  Thus, we expect that 
$p\equiv N-1 \mod \ell$ with probability $1/(\ell-1)$, provided that $N-1$ is 
not divisible by~$\ell$.  For fixed $\ell$, this simply adds a constant factor 
to the expected running time.

For $p>3$, the added restriction that $E$ be ordinary is equivalent to 
demanding that $p\ne N-1$.  Excluding this single value of $p$ does not change
the expected running time of the algorithm. 
\end{proof}

\section{Genus-2 curves and Jacobians}
\label{S:genus2}

\noindent
Let $q$ be a power of a prime.  A polynomial $f\in \ZZ[T]$ is called a 
\emph{Weil $q$-polynomial} if there is an abelian variety $A$ over $\FF_q$ such
that $f$ is the characteristic polynomial of the Frobenius endomorphism 
$\Phi_q \in\End A$ acting on the Tate module $T_\ell(A)$ for some prime 
$\ell\nmid q$. As the complex roots of a Weil $q$-polynomial have absolute 
value~$\sqrt q$ and come in $g$ complex conjugate pairs, with $g$ the dimension
of $A$, the quartic Weil $q$-polynomials arising in genus 2 have the form 
\begin{align}
\label{eq:Weilpolg=2}
 f &= T^4 - a T^3 + (b + 2q) T^2 - a q T + q^2 \\
\notag
   &= (T^2 + q)^2 - a T (T^2 + q) + b T^2,
\end{align}
with $a, b\in\ZZ$ satisfying the inequalities
\begin{equation}
\label{eq:wedge-ineq}
2|a|\sqrt{q} - 4q \le b 
                  \le \frac{1}{4} a^2
                  \le 4q.
\end{equation}
These inequalities define a wedge-shaped region inside the rectangle in the
$(a,b)$-plane defined by $|a|\le 4\sqrt q$ and $|b|\le 4q$, and it is natural
to ask which pairs $(a,b)$ of integers satisfying the inequalities 
\eqref{eq:wedge-ineq} come from the Weil $q$-polynomial of an abelian surface,
or from the Weil $q$-polynomial of the Jacobian of a genus-2 curve.  The 
Honda--Tate theorem~\cite{Tate}*{Th\'eor\`eme~1} can be used to determine 
the pairs $(a,b)$ that come from abelian surfaces, and \cite{HNR}*{Theorem~1.2}
explains how to determine which $(a,b)$ come from Jacobians of curves.  For our 
purposes, it will be sufficient to note that all pairs of integers $(a, b)$ 
satisfying the inequalities \eqref{eq:wedge-ineq} \emph{and} the coprimality
condition $\gcd(b,q)=1$ arise from the coefficients of the Weil $q$-polynomial 
of an abelian surface over $\FF_q$ --- in fact, an ordinary abelian surface.

Let $C$ be a curve of genus 2 defined over $\FF_q$, and let $J=\Jac C$ be its
Jacobian, so that $J$ is an abelian surface defined over~$\FF_q$. Let $f$ be 
the Weil polynomial of~$J$, with coefficients given as 
in~\eqref{eq:Weilpolg=2}.  The pair $(\#C(\FF_q), \#J(\FF_q))$ of orders over 
$\FF_q$ is determined by~$f$, and conversely. In concrete terms, we have
\begin{align}
\label{eq:curveorder}
\#C(\FF_q)&=q+1-a    \\
\label{eq:jacorder}
\#J(\FF_q)&=f(1)=(q+1)^2-a(q+1)+b. 
\end{align}
It follows that the order $\#J(\FF_q)$ lies in the genus-2 Hasse-interval
\[
\calH_q^{(2)}=[(\sqrt q-1)^4, (\sqrt q+1)^4]
\]
forming the analogue of \eqref{eq:Hasseqg=1}.  The interval $\calH_q^{(2)}$
has length $8\sqrt q (q+1)$ and is centered around $q^2+6q+1$. We have an 
equivalence
\[
N\in \calH_q^{(2)} \quad\Longleftrightarrow\quad
q\in \calH_{\sqrt N}=[(N^{1/4}-1)^2, (N^{1/4}+1)^2]
\]
that is not as symmetric as \eqref{eq:Nqsymmgenus1} in $N$ and $q$.  This is
because the order $N$ of an abelian surface over $\FF_q$ has order of 
magnitude~$q^2$, not~$q$. 

Just as in the elliptic case, the union of the integers in the genus-$2$ Hasse
intervals $\calH_q^{(2)}$ for the fields $\FF_q$ that are \emph{not} prime 
fields forms a zero density subset of all positive integers.  Our inability to
prove prime gap bounds as in \eqref{eq:gapbound} prevents us also in this case
from showing that every positive integer arises as the order of an abelian 
surface over a finite field. However, we can prove with some extra effort that,
just as in the elliptic case, the (conjecturally empty) set of integers $N$ 
that do \emph{not} arise as the order of an abelian surface forms a very thin 
subset of all positive integers.

\begin{theorem}
\label{thm:zerodensity}
The set of positive integers $N\le X$ that do not occur as the order of an 
abelian surface over a finite field has cardinality $O(X^{5/6})$ for 
$X\to\infty$.
\end{theorem}

The proof relies on a lemma about the central part
\[
\calC_q^{(2)}= [(q+1)^2-q^{3/2}, (q+1)^2+q^{3/2}] 
\]
of the genus-$2$ Hasse interval $\calH_q^{(2)}$.

\begin{lemma}
\label{lem:interval}
If $q$ is prime, then every integer in $\calC_q^{(2)}$ is of the form $f(1)$
for the Weil $q$-polynomial $f$ of some abelian surface over $\FF_q$.
\end{lemma}

\begin{proof}
Let $N$ be an integer in $\calC_q^{(2)}$, and write $N = (q+1)^2 + m$, so that 
$|m|\le q^{3/2}$.  We would like to find integers $a$ and $b$ satisfying the
inequalities~\eqref{eq:wedge-ineq} such that we also have $m = -a(q+1) + b$ and
$\gcd(b,q) = 1$; then the polynomial $f$ defined by 
Equation~\eqref{eq:Weilpolg=2} will be the Weil $q$-polynomial of an ordinary
abelian surface over $\FF_q$, and $N = f(1)$.

Define three pairs of integers $(a_0,b_0)$, $(a_1, b_1)$, and $(a_2,b_2)$ by
setting
\begin{align*}
a_0 &= -\lfloor m/(q+1)\rfloor  & b_0 &= m +a_0(q+1)\\
a_1 &= a_0 - 1                  & b_1 &= m +a_1 (q+1)\\
a_2 &= a_0 - 2                  & b_2 &= m +a_2 (q+1).
\end{align*}
We claim that if $q>7$ then at least one of these pairs $(a_i,b_i)$ satisfies 
the inequalities~\eqref{eq:wedge-ineq} and has $\gcd(b_i,q) = 1.$

First note that the inequality $|m|\le q^{3/2}$ gives
\begin{align*}
- q^{1/2}     &< a_0 < q^{1/2} + 1 &     0 & \le b_0 \le q \\
- q^{1/2} - 1 &< a_1 < q^{1/2}     &  -q-1 & \le b_1 \le -1\\
- q^{1/2} - 2 &< a_2 < q^{1/2} - 1 & -2q-2 & \le b_2 \le -q-2.
\end{align*}
It is easy to check that if $q>7$ then $(a_1,b_1)$ 
satisfies~\eqref{eq:wedge-ineq}, so if $\gcd(b_1,q) = 1$ we are done.  Thus, to 
prove our claim we may assume that we are in the case where
$\gcd(b_1,q) \neq 1$.  Since $q$ is prime, we must have $b_1 = -q$.  Therefore 
$b_0 = 1$ and $b_2 = -2q - 1$.

Since $b_0= 1$ we clearly have $\gcd(b_0,q)=1$.  We check that the only way 
$(a_0,b_0)$ will fail to satisfy~\eqref{eq:wedge-ineq} is if $a_0^2 < 4$.
Thus, if $(a_0,b_0)$ does not satisfy the desired conditions, it must be the
case that $a_0\in\{-1,0,1\}$, from which it follows that $a_2\in\{-3,-2,-1\}$.  
To finish the proof of our claim, we may assume we are in this case.

Since $b_2 = -2q-1$ we have $\gcd(b_2,q) = 1$, and it is easy to check that 
when $q>7$ and $|a_2|\le 3$, the pair $(a_2, b_2)$ 
satisfies~\eqref{eq:wedge-ineq}.  This proves our claim, and shows that the 
lemma holds for $q>7$.

It remains to verify the lemma for primes $q\le 7$.  By hand or machine, it is
not hard to check that for all of the relevant values of $N$ it is \emph{still}
the case that one of the pairs $(a_i,b_i)$ defined above satisfies the 
inequalities~\eqref{eq:wedge-ineq} and has $\gcd(b_i,q) = 1$, with exactly five
exceptions: the cases where $(q,N)$ is one of $(2, 10)$, $(3, 17)$, $(3, 21)$,
$(5, 43)$, or $(7, 73)$.  For these cases, we can take $f$ to be the Weil 
polynomial
$x^4 +   x^3 + 2 x^2 +  2 x +  4$,
$x^4 +   x^3 + 3 x^2 +  3 x +  9$,
$x^4 + 2 x^3 + 3 x^2 +  6 x +  9$,
$x^4 + 2 x^3 + 5 x^2 + 10 x + 25$, or
$x^4 + 2 x^3 + 7 x^2 + 14 x + 49$,
respectively.
\end{proof}

\begin{remark}
With more effort, one can show that the only prime powers $q$ for which the
conclusion of Lemma~\ref{lem:interval} fails to hold are the nonprime prime 
powers $q\le 81$.
\end{remark}

We will also use a slight variant of a result of Matom\"aki~\cite{Matomaki}.

\begin{lemma}
\label{L:Matomaki}
Let $p_n$ denote the $n$-th prime number and let $d_n = p_{n+1}-p_n$ denote the
$n$-th prime gap. For every $c > 1/\sqrt{2}$ there is a constant $B>0$ such that 
\[
\sum_{\substack{d_n > c\sqrt{p_n} \\ \strut p_n \le X}} d_n < B X^{2/3}
\]
for all $X>0$.
\end{lemma}

\begin{proof}
If the lemma is true for a given value of $c$ then it is true for all larger 
values, so it suffices to consider the case $c<1$.

Theorem~1.1 (p.~489) of~\cite{Matomaki} states that there is a constant $A>0$
such that for all $x$ we have
\[
\sum_{\substack{d_n \ge \sqrt{x}\\ \strut x\le p_n\le 2x}}  d_n < A x^{2/3}.
\]
Let $b = 1/c^2$, and note that $1 < b < 2$.  Suppose $p_n\le X$ satisfies 
$d_n > c\sqrt{p_n}$, and let $i$ be the unique nonnegative integer such that 
$p_n$ lies in the half-open interval $I_i := \left((b/2)^{i+1} X, (b/2)^i X\right]$.  
Set $x = b^i X / 2^{i+1}$, so that $I_i = (b x, 2x]$. Since $p_n > bx$, we have 
$c \sqrt{p_n} > \sqrt{x}$, so
\[
\sum_{\substack{d_n > c\sqrt{p_n} \\ \strut bx \le p_n \le 2x}} d_n
\le
\sum_{\substack{d_n > \sqrt{x} \\ \strut bx \le p_n \le 2x}} d_n 
\le
\sum_{\substack{d_n > \sqrt{x} \\ \strut x \le p_n \le 2x }} d_n < A x^{2/3}.
\]
The interval $(1,X]$ is the union of the intervals $I_i$, so
\[
\sum_{\substack{d_n > c\sqrt{p_n} \\ \strut p_n \le X}} d_n 
<
\sum_{i\ge 0} A \left(\frac{b^i}{2^{i+1}} X\right)^{2/3}
= B X^{2/3},
\]
where 
$B = A/(2^{2/3} - b^{2/3}).$
\end{proof}

\begin{proof}[Proof of Theorem~\textup{\ref{thm:zerodensity}}]
By Lemma~\ref{lem:interval}, if $N\le X$ is an integer that is not the order of
an abelian surface over a finite field, then $N$ does not satisfy 
$|N-(p+1)^2|< p^{3/2}$ for a single prime~$p$.  If $p_n$ and $p_{n+1}$ are 
consecutive primes for which we have $(p_n+1)^2 < N < (p_{n+1}+1)^2$, then 
$p_n < Y := \sqrt{X}$, and we have
\[
(p_{n+1}+1)^2- (p_n+1)^2 = (p_{n+1}-p_n) (p_{n+1}+p_n+2) 
 > (p_{n+1})^{3/2}+p_n^{3/2}.
\]
It follows that the prime gap $d_n$ satisfies 
\[
d_n > \frac{(p_{n+1}/p_n)^{1/2} p_{n+1} + p_n}{p_{n+1} + p_n + 2} \sqrt {p_n}
    > (5/7) \sqrt{p_n}.
\]
The number of $N\le X$ that are not orders of abelian surfaces is therefore at
most the total length $\sum d_n \cdot (p_{n+1}+p_n+2)$ of those intervals 
$[(p_n+1)^2, (p_{n+1}+1)^2]$ for which $p_n < Y$ and $d_n>(5/7) \sqrt{p_n}$.

Lemma~\ref{L:Matomaki} shows that the sum $\sum d_n$ over all $n$ for
which $p_n < Y$ and $d_n>(5/7) \sqrt{p_n}$ is bounded by $O(Y^{2/3})=O(X^{1/3}).$ 
The $p_n$ are all bounded by $X^{1/2}$, so the sum 
$\sum d_n \cdot (p_{n+1}+p_n+2)$ is bounded by $O(X^{5/6})$.
\end{proof}

\begin{remark}
Just as for elliptic curves, it is a safe conjecture that every positive 
integer occurs as the order of an abelian surface over a finite field. This may
be very hard to \emph{prove}, but there is no practical obstruction in showing
any given integer to be the order of an abelian surface, as it will usually 
arise as $f(1)$ for many quartic Weil $q$-polynomials $f$.
\end{remark}

Any product $f=f_1\cdot f_2$ of two elliptic Weil $q$-polynomials 
$f_i = T^2 - t_i T + q \in\ZZ[T]$ is a genus-$2$ Weil $q$-polynomial.  It
corresponds to the class of abelian surfaces isogenous to the product
$E_1\times E_2$ of elliptic curves $E_i$ with Weil polynomial~$f_i$.  If the 
Jacobian of a genus-2 curve $C$ is in this class, $C$ is said to have 
\emph{split Jacobian}.  In this split case, the order of the Jacobian factors
as
\[
\#J(\FF_q)=\#E_1(\FF_q)\cdot \#E_2(\FF_q),
\]
whereas the curve itself has order 
\begin{equation}
\label{EQ:traces}
\#C(\FF_q)= q+1-t_1-t_2.
\end{equation}
The explicit construction of curves $C$ from $E_1$ and $E_2$ is the topic of 
Section~\ref{S:gluing}.

As for genus~$1$, it is possible to construct abelian surfaces over $\FF_q$ 
with a given quartic Weil $q$-polynomial $f\in\ZZ[X]$ as Jacobians of explicit
genus-$2$ curves using \emph{complex multiplication} methods.  In the most
interesting case where $f$ is irreducible, $K=\QQ[X]/(f)$ is a quartic CM-field
and $\calO=\ZZ[X]/(f)$ an order in $K$. One then wants to find an abelian
surface $A/\FF_q$ for which the subring $\ZZ[\Phi_q]\subset\End A$ generated
by the Frobenius $\Phi_q$ of $A$ is isomorphic to $\calO$.  As in the elliptic 
case, this is done by considering abelian surfaces over the complex numbers 
admitting CM by the order~$\calO$, and even by the full ring of integers 
$\calO_K\supseteq \calO$ of $K$. Such complex abelian surfaces are quotients of
$\CC^2$ modulo suitably embedded $\calO_K$-ideals, and their isomorphism class
is characterized by three absolute \emph{Igusa invariants}, just like the 
isomorphism class of a complex elliptic curve $\CC/I$ is characterized by the
absolute $j$-invariant of the lattice $I$.

In the elliptic case, the isomorphism classes of the curves $\CC/I$ having CM 
by an imaginary quadratic order correspond to the ideal classes of that order, 
and their $j$-invariants form the roots of the Hilbert class polynomial 
\eqref{eq:Hilbpol}, which lies in $\ZZ[X]$. In a similar way, the three Igusa 
invariants of the relevant $\calO_K$-ideal classes form the roots of three 
polynomials $H_{i, K}\in\QQ[X]$, $i=1, 2, 3$. They are known as the
\emph{Igusa class polynomials} of the quartic field $K$, and computing them is 
the key step in any CM-algorithm.

Once one has found the Igusa class polynomials $H_{i, K}$, one can reduce these
modulo $p=\characteristic(\FF_q)$ to find the Igusa invariants of abelian 
surfaces $J$ over $\FF_q$ having CM by $\calO_K$.  Up to twisting over~$\FF_q$,
these have the desired Weil $q$-polynomial. Actual equations of abelian 
surfaces cannot easily be given, but an algorithm of Mestre~\cite{Mestre}
allows us to write down an explicit genus-$2$ curve $C$ having a Jacobian $J$ 
with the desired Igusa invariants.  This allows us to do actual computations in
the group $J(\FF_q)$, in terms of divisors on $C$.

There are myriad details that go into a full explanation of the genus-$2$
CM-method, and of the way one can proceed algorithmically. A detailed account 
that includes the first complete run time analysis was given by
Streng~\cite{Streng}. All we will need in Section \ref{S:Jacobians} is that the
run time of a CM-algorithm to produce genus-$2$ curves $C$ with irreducible
Weil polynomial $f\in\ZZ[X]$ is necessarily exponential in the size 
$\log \Delta_K$ of the discriminant $\Delta_K$ of $K=\QQ[X]/(f)$.  This is
because the degree of the Igusa class polynomials that occur in the algorithm 
grows like a positive power of $\Delta_K$, as follows.

\begin{proposition}
\label{prop:igusadegree}
The degree $n_K$ of the Igusa class polynomials of a quartic CM-field $K$
satisfies $n_K \gg \Delta_K^{1/4-\eps}$ for all $\eps >0$.
\end{proposition}

\begin{proof}
By \cite{Streng}*{Lemma~4.14}, the degree $n_K$ of the Igusa class polynomials 
of~$K$ is, up to a factor $1$ or~$2$, equal to the relative class number 
$h_K^-= h_K/h_K^+$ of~$K$.  Here $h_K$ and $h_K^+$ denote the class numbers of
$K$ and its real quadratic subfield~$K^+$. 
In \cite{Louboutin}*{Corollaries~29 and~32}, we find the Brauer--Siegel type
result that the logarithm of $h_K^-$ is asymptotic to 
$\frac{1}{2}\log  (\Delta_K^-)$, with $\Delta_K^-=\Delta_K/\Delta_K^+$ the 
quotient of the discriminants of $K$ and~$K^+$. As we have 
$\Delta_K=(\Delta_K^+)^2 \cdot M$, with $M\in\ZZ_{>0}$ the absolute norm of the
relative discriminant of $K$ over $K^+$, we see that 
$\Delta_K^- = \Delta_K^+\cdot M$ is a divisor of $\Delta_K$ exceeding
$\Delta_K^{1/2}$, whence $(\Delta_K^-)^{1/2}\ge \Delta_K^{1/4}$.  The result
follows.
\end{proof}

\section{Genus-2 Jacobians of given order}
\label{S:Jacobians}

\noindent
We now give a proof of Theorem \ref{Th:J}.  The statement of the theorem is 
that, in order to realize all integers $N$ in the interval $[1 , X]$ as orders
of genus-$2$ Jacobians over finite fields, we will necessarily encounter Weil
polynomials generating quartic CM-fields of discriminant exceeding any
prescribed constant multiple of $\sqrt X$, provided that $X$ is sufficiently
large.

All \emph{Weil polynomials} in this Section will be Weil $q$-polynomials of
abelian surfaces, that is, quartic polynomials $f\in\ZZ[T]$ of the form 
\eqref{eq:Weilpolg=2} arising as the characteristic polynomial of the Frobenius
endomorphism acting on the Tate module of an abelian surface defined 
over~$\FF_q$. If $f$ is such a Weil polynomial and $f(1)$ an integer in the
interval $[1 , X]$, then the inequalities $(\sqrt q -1)^4\le f(1)\le X$ imply 
that we have a bound
\begin{equation}
\label{eq:qbound}
\sqrt q\le X^{1/4}+1
\end{equation}
for the square root of the prime power $q$ involved.

We begin by showing that for large $X$, reducible Weil polynomials only account
for very few orders of abelian surfaces in the range $[1 , X]$.

\begin{proposition}
The number of positive integers $N\le X$ arising as the value $f(1)$ of a
reducible quartic Weil polynomial $f\in\ZZ[T]$ is $O(X^{3/4})$ for 
$X\to\infty$.
\end{proposition}

\begin{proof}
Suppose $f\in\ZZ[T]$ is a reducible quartic Weil $q$-polynomial. Since the real
roots of a Weil polynomial occur with even multiplicity, the polynomial $f$ is
either equal to a product $f=g_1g_2$ of two quadratic polynomials 
$g_1, g_2\in\ZZ[T]$ with complex conjugate roots of absolute value $\sqrt q$,
or it is equal to $(x^2 - q)^2$, in which case we write $f = g_1 g_2$ with 
$g_1 = g_2 = -x^2 + q$.  In both cases, we see that the value $N=f(1)$ is the
product of the integers $g_1(1)$ and $g_2(1)$ in the elliptic Hasse interval
$\calH_q$ defined in~\eqref{eq:Hasseqg=1}.

We write $N=g_1(1)\cdot g_2(1)=(s+t)(s-t)$, with $s=(g_1(1)+g_2(1))/2$ a 
half-integer lying in $\calH_q$, and $t=|g_1(1)-g_2(1)|/2$ a nonnegative 
half-integer of absolute value at most $2\sqrt q$. By \eqref{eq:Hasseqg=1} 
and \eqref{eq:qbound}, the positive integer $2s$ can be bounded by 
\[
2s \le 2(\sqrt q +1)^2 \le 2(X^{1/4}+2)^2,
\]
whereas $2t$ is a nonnegative integer not exceeding $4(X^{1/4}+1)$. As the 
integers $2s$ and $2t$ determine $N$, we see that for every $\eps > 0$, no 
more than $(8+\eps)X^{3/4}$ values of $N$ occur in $[1, X]$, when $X$ is 
sufficiently large.
\end{proof}

\begin{corollary}
\label{cor:split}
The integers $N$ arising as the value $f(1)$ of a reducible genus-$2$ Weil
polynomial $f\in\ZZ[T]$ form a zero-density subset of all positive integers.
\qed
\end{corollary}

We can now focus on \emph{irreducible} Weil polynomials $f$, which have the
property that a root of $f$ generates a quartic CM-field $K=\QQ[T]/(f)$ 
over~$\QQ$. Given~$K$, the number of such $f$ can be bounded in the following
way; compare to \cite{EL}*{Prop.~4}.

\begin{proposition}
\label{prop:howeslemma}
Let $K$ be a quartic CM-field having $w_K$ roots of unity, and~$q$ a prime
power. Then there are at most $2w_K$ irreducible quartic Weil $q$-polynomials
having a root in $K$.
\end{proposition}

\begin{proof}
Let $q$ be a power of a prime $p$, and $\pi\in K$ a quartic Weil $q$-number,
that is, an algebraic integer $\pi$ of degree four with 
$|\varphi(\pi)| = \sqrt{q}$ for all complex embeddings
$\varphi\colon K\to\CC$. Then we have $\pi\pibar = q$, where $x\mapsto\bar{x}$ 
denotes conjugation over the maximal real subfield $\Kp$ of $K$. By the 
Honda--Tate theorem~\cite{Tate}, the dimension of the abelian varieties in 
the isogeny class associated to $\pi$ can be read off from properties of the
principal ideal $\gotha = (\pi)$, which satisfies $\gotha\gothabar = (q)$ and 
is only divisible by primes lying over~$p$. In particular, the conjugacy class
of $\pi$ corresponds to an isogeny class of abelian surfaces if and only if we
have           
\begin{equation}                                                   
\label{EQ:HT}                                                     
\frac{f_\gothp \ord_\gothp \gotha}{\ord_\gothp q}\in\ZZ         
\end{equation}                                                                  
for every prime $\gothp$ of $K$ lying over $p$. Here $f_\gothp$ denote the
residue class degree of $\gothp$.

We first show that there are at most four integral ideals $\gotha$ of $K$        
satisfying $\gotha\gothabar = (q)$ for which~\eqref{EQ:HT} holds for all 
primes $\gothp$ of $K$ over~$p$.         

Suppose $\gotha$ is such an ideal.  Let $\gothp$ be a prime of $K$ lying 
over~$p$.  If $\gothp = \gothpbar$ then the condition $\gotha\gothabar = (q)$
shows that $2 \ord_\gothp \gotha = \ord_\gothp q$, so the order of $\gotha$ at
$\gothp$ is determined.  On the other hand, suppose $\gothp \neq \gothpbar$.
Then $f_\gothp\le 2$, and from $\gotha\gothabar = (q)$ we see that 
\[
\ord_\gothp\gotha + \ord_{\gothpbar}\gotha = \ord_\gothp q.
\]             
Thus, from~\eqref{EQ:HT} we see that either                                  
\begin{enumerate}                                                            
\item[(a)] $\ord_\gothp\gotha = 0$ and 
           $\ord_{\gothpbar}\gotha = \ord_\gothp q$, or 
\item[(b)] $\ord_\gothp\gotha = \ord_\gothp q$ and 
           $\ord_{\gothpbar}\gotha = 0$, or
\item[(c)] $f_\gothp = 2$ and 
           $\ord_\gothp\gotha = \ord_{\gothpbar}\gotha = (1/2)\ord_\gothp q$.  
\end{enumerate}                                                               
In short, there is one possibility for $\ord_\gothp\gotha$ if $\gothp$ is 
ramified or inert in $K/\Kp$, there are two possibilities for the pair 
$(\ord_\gothp\gotha,\ord_{\gothpbar}\gotha)$ if $\gothp$ splits in $K/\Kp$ and
lies over a prime of $\Kp$ with residue class field degree~$1$, and there are
at most three possibilities for the pair 
$(\ord_\gothp\gotha,\ord_{\gothpbar}\gotha)$ if $\gothp$ splits in $K/\Kp$ and
lies over a prime of $\Kp$ with residue class field degree~$2$.  By considering
the various ways $p$ can split in $K$, we find that there are at most four
possibilities for the vector of valuations of $\gotha$ at the primes over~$p$,
so there are no more than four ideals $\gotha$ with $\gotha\gothabar = (q)$ and
such that~\eqref{EQ:HT} holds for all primes $\gothp$ of $K$ over~$p$.   

Suppose such an ideal $\gotha$ is generated by a Weil number~$\pi_0$. If 
$\gotha$ is also generated by another Weil number~$\pi$, then $\pi/\pi_0$ is a
unit of~$K$, and $\varphi(\pi/\pi_0)$ lies on the unit circle for all 
embeddings $\varphi$ of $K$ into the complex numbers.  It follows that 
$\pi = \zeta\pi_0$ for some root of unity~$\zeta$.  Therefore, if $\gotha$ can 
be generated by any Weil numbers at all, it can be generated by exactly $w_K$
of them. Thus, the number of Weil $q$-numbers in $K$ is at most $4 w_K$.  

Suppose $f$ is an irreducible quartic Weil $q$-polynomial with a root $\pi$ 
in~$K$.  Then $\pibar$ is also a root of $f$ in $K$, and $\pi\ne\pibar$ because
$\pi$ is a root of an irreducible quartic and hence not an element of the real
subfield of $K$.  Thus, every irreducible quartic Weil $q$-polynomial with a
root in $K$ produces at least two distinct Weil numbers in~$K$, so the number
of such polynomials with roots in $K$ is at most $2 w_K$.
\end{proof}

\begin{corollary}
\label{cor:weilpolsnumber}
Let $K$ be a quartic CM-field.  Then the number of irreducible genus-$2$ Weil
$q$-polynomials with  a root in $K$ that satisfy the bound 
\[
\sqrt q\le X^{1/4}+1 \leqno{\eqref{eq:qbound}}
\]
is at most $50 X^{1/2}/\log X$ for $X$ sufficiently large.
\end{corollary}

\begin{proof}
It is easy to see that the number of integers less than $y$ of the form $a^n$
with $n>1$ is less than $\sqrt{y}\log_2 y$.  Combining this fact with the prime
number theorem, we find that the number of prime powers less than $y$ is 
asymptotic to $y / \log y$.  It follows that the number of prime powers $q$ up 
to $(X^{1/4}+1)^2$ is less than $(2+\eps) X^{1/2}/\log X$, for $X\gg_\eps 0$.
From~Lemma~\ref{prop:howeslemma} we see that for each of these $q$ there are at
most $2 w_K$ irreducible quartic Weil $q$-polynomials with a root in~$K$. 
Since $w_K\le12$ for quartic fields~$K$, the corollary follows.
\end{proof}

Now that we know an upper bound on the number of Weil polynomials `coming from'
a fixed quartic CM-field $K$, we still need a result that expresses the fact
that there are not too many quartic CM-fields of small discriminant.

\begin{proposition}
\label{proposition:cohen}
For $B\in\ZZ_{>0}$ a sufficiently large integer, the number of isomorphism 
classes of quartic CM-fields of discriminant at most $B$ is bounded by~$B$.
\end{proposition}

\begin{proof}
If $K$ is a quartic CM-field, then the Galois group over $\QQ$ of its normal
closure is the dihedral group $D_4$ of order~$8$, the cyclic group $C_4$ of 
order~$4$, or the Klein four group $V_4=C_2\times C_2$. It follows from the
results of Cohen et al.~\cite{CDO} that the number of isomorphism classes of 
quartic CM-fields $K$ of discriminant $\Delta_K\le B$ with group $D_4$ is 
asymptotically equal to $c\cdot B$, where $c\approx.05$ is some explicit real
constant.  As the number of isomorphism classes of quartic fields $K$ with
groups $C_4$ and $V_4$ and bounded discriminant $\Delta_K\le B$ is 
asymptotically much smaller, and grows \cite{CDO}*{Section 1.1} like 
$c'\cdot B^{1/2}$ and $c''\cdot B^{1/2}(\log B)^2$ for certain explicit 
positive constants $c', c''$, the result follows.
\end{proof}

After these preparations, the proof of Theorem~\ref{Th:J} is more or less 
straightforward.

\begin{proof}[Proof of Theorem~\textup{\ref{Th:J}}]
Suppose that 
\[
\limsup_{N\to\infty} \frac{\Delta(N)}{\sqrt N}
\]
assumes a finite value. Then there exist a constant $C>0$ such that we have 
\begin{equation}
\label{eq:absurd}
\Delta(N)\le C \sqrt N \qquad\text{for all integers $N>0$.}
\end{equation}
Let $\mathbf{A}\subset \ZZ_{>0}$ be the subset of integers that are not the 
value $f(1)$ of any genus-$2$ Weil polynomial $f$, or the value $f(1)$ of a
reducible genus-$2$ Weil polynomial~$f$.  Then $\mathbf{A}$ is a zero density 
subset by Theorem \ref{thm:zerodensity} and Corollary \ref{cor:split}.

For all integers $N\notin \mathbf{A}$, the minimal discriminant $\Delta(N)$ is 
the discriminant of a quartic CM-field $\QQ[T]/(f_N)$, with $f_N$ an 
irreducible quartic Weil polynomial satisfying $N=f_N(1)$. If we take $X$
sufficiently large, then there are $(1-\eps)X$ integers $N$
lying in $[1,X]\setminus\mathbf{A}$, with $\eps>0$ small. Moreover, among the 
CM-fields $\QQ[T]/(f_N)$ that occur for these integers, there will be at least
$(1/50) (1-\eps) X^{1/2}\log X$ pairwise nonisomorphic fields, as a single 
isomorphism class will yield no more than $50 X^{1/2}/\log X$ polynomials
$f_N$ by Corollary~\ref{cor:weilpolsnumber}.  By 
Proposition~\ref{proposition:cohen}, we will find values
\[
\Delta(N)\ge \frac{1}{50} (1-\eps) X^{1/2}\log X
\]
among $N\in [1,X]$ for $X\gg_\eps 0$.  This contradicts \eqref{eq:absurd}.
\end{proof}

As we state it, Theorem~\ref{Th:J} leaves open the possibility that there is
only a very thin set of integers $N$ on which $\Delta(N)/\sqrt N$ is unbounded.
This is however not the case, as an easy adaptation of the proof shows.

\begin{corollary}
\label{cor:generic}
Let $S$ be a set of positive integers of positive density. Then the minimal
genus-$2$ Jacobian discriminant in Theorem $\ref{Th:J}$ satisfies
\[
\limsup_{N\in S, N\to\infty} \frac{\Delta(N)}{\sqrt N}=+\infty.
\]
\end{corollary}

\begin{proof}
For any set $S$ of positive density, the number of pairwise nonisomorphic
CM-fields $\QQ[T]/(f_N)$ encountered (as in the preceding proof) for $N$ on 
the set $S\setminus \mathbf{A}$ of positive density will still be at least 
$c X^{1/2}\log X$ for some $c>0$.
\end{proof}

\begin{remark}
The single factor of $\log x$ that makes the proof of Theorem \ref{Th:J} work 
suggests that the theorem is rather sharp, but in reality it is not. In fact, 
Corollary~\ref{cor:weilpolsnumber} is far from optimal, as is does not take 
into account that the existence of Weil $q$-numbers in a quartic CM-field $K$
not only implies that the rational prime divisor of $q$ has a certain splitting
behavior in $K$, as indicated in the proof of
Proposition~\ref{prop:howeslemma}, but also that the ideal $\gotha$ occurring
in the proof is \emph{principal}. In view of the growth of class numbers
with~$\Delta_K$, these are serious restrictions that we simply disregarded.
\end{remark}

By a \emph{CM-construction} for genus-2 Jacobians of prescribed order~$N$ we
mean \emph{any} algorithm that, in order to find a curve $C$ over $\FF$ with 
Jacobian of order~$N$, writes down\footnote{
  By `writing down' a polynomial in $x$ of degree $n$, we mean writing down
  the coefficient of each monomial $1, x, \ldots, x^n$ separately, even the 
  ones that happen to be zero.  Thus, in this reckoning, it takes time $n+1$
  to write down the polynomial $x^n - 1$.
}
the Igusa class polynomials of a quartic CM-field $K$ such that a curve in
characteristic zero with CM by $K$ reduces to $C$ over $\FF$.
Theorem~\ref{Th:J} implies that such constructions will necessarily have 
exponential run time.

\begin{corollary}
\label{cor:runtime}
Any CM-construction for genus-$2$ Jacobians of prescribed order~$N$ will have
an exponential run time, of order at least $N^{1/8 -\eps}$ for all $\eps >0$.
\end{corollary}

\begin{proof}
As the discriminants of the CM-fields involved grow at least as fast as
$\sqrt N$, by Theorem \ref{Th:J}, the Igusa class polynomials involved will be 
of degree at least $N^{1/8 -\eps}$ by Proposition~\ref{prop:igusadegree}, so 
writing them down takes at least time $N^{1/8 -\eps}$.
\end{proof}

The lower bound in the corollary is rather weak, as it does not account for 
the size of the coefficients of Igusa polynomials. These also appear to grow as
a positive power of the discriminant~\cite{Streng}, but we have no good
\emph{proven} lower bounds for the total length of the coefficients. However,
the lower bound on the degree of the Igusa class polynomials that we do use 
shows that any algorithm that requires writing down even just the 
\emph{reduction} of an Igusa class polynomial modulo some auxiliary prime will 
necessarily take exponential time.

\section{Gluing elliptic curves}
\label{S:gluing}

\noindent
Our construction of genus-$2$ curves with a given number of points depends on
methods of constructing genus-$2$ curves that have Jacobians isogenous to a 
product of given elliptic curves.  In this section we will present two
algorithms for producing such curves. The first is simply an algorithmic 
description of an explicit construction given in~\cite{HLP}; the second is
based on an explicit construction given in Appendix~\ref{S:appendix}.

As is explained in~\cite{FK}, every genus-$2$ curve $C$ with a nonsimple 
Jacobian arises (perhaps in several ways) from specifying two elliptic curves
$E_1$ and $E_2$, an integer $n>1$, and an isomorphism  
$\psi\colon E_1[n]\to E_2[n]$ of the $n$-torsion subschemes of $E_1$ and $E_2$
that is an anti-isometry with respect to the Weil pairing.  More precisely, 
there is an isogeny $\varphi$ from $E_1\times E_2$ to the Jacobian $\Jac C$ of
$C$ whose kernel is the graph of the isomorphism~$\psi$, and the pullback via 
$\varphi$ of the canonical polarization of $\Jac C$ is equal to $n$ times the
product polarization on $E_1\times E_2$.  In this situation, we say that 
$\Jac C$ (or, by an abuse of language that we will find convenient, $C$ itself)
is obtained by \emph{gluing} $E_1$ and $E_2$ together along their $n$-torsion 
subgroups via~$\psi$.

The relationship between $C$ and $E_1$ and $E_2$ can also be summarized by
saying that there are minimal degree-$n$ maps $\varphi_i\colon C\to E_i$ such
that $\varphi_{2*}\varphi_1^{*} = 0$; here \emph{minimal} means that 
$\varphi_i$ does not factor through a nontrivial isogeny.  Given $E_1$, $E_2$,
$n$, $\psi$, $C$, and $\varphi$ from the preceding paragraph, one obtains
$\varphi_i$ by composing an embedding of $C$ into its Jacobian with the dual
isogeny $\hat{\varphi}\colon \Jac C\to E_1\times E_2$, and then projecting
$E_1\times E_2$ onto~$E_i$.  Conversely, given minimal maps $\varphi_1$ and 
$\varphi_2$, one takes $\varphi$ to be the degree-$n^2$ isogeny 
\[
\varphi_1^*\times\varphi_2^* \colon E_1\times E_2\to \Jac C,
\]
and notes that the kernel of $\varphi$ is the graph of an anti-isometry 
$\psi\colon E_1[n]\to E_2[n]$.

Over the complex numbers, the full family of genus-$2$ curves arising from the
case $n=2$ was given in 1832 by 
Jacobi~\citelist{\cite{Jacobi} 
                 \cite{Jacobi:Werke}*{Volume~I, pp.~373--382}}
as a postscript to his review of Legendre's 
\emph{Trait\'e des fonctions elliptiques}~\cite{Legendre}, and in 1885
Goursat~\cite{Goursat}*{Exemple~II, pp.~155--157} gave a family for $n=3$
that misses only a single curve.  We will use more recent references because we
must keep track of fields of definition, but the formulas we use can be traced
back to these early works.

For our intended applications we will be concerned only with the case of curves
over finite fields, but the algorithms will work --- and will run in 
polynomial time --- over any field $k$ in which elements can be precisely 
specified and in which arithmetic can be done in polynomial time. We will use 
the term \emph{computationally amenable} to describe such fields~$k$. It is
easy to see that any finitely-generated extension of a prime field is 
computationally amenable; the principal examples of such fields that we will 
have in mind are finite fields and number fields.

In fact, our gluing algorithms are based on solving systems of polynomial
equations, so the constructions underlying them work over other fields as well;
for example, the complex numbers.  We could phrase almost all of our results in
terms that Jacobi, Legendre, and Goursat would be familiar with, but since we
do want to speak about polynomial time algorithms, we restrict ourselves to 
computationally amenable fields.

First we give an algorithm that produces the list of all genus-$2$ curves that
can be obtained by gluing two given elliptic curves along their $2$-torsion
subgroups; the algorithm is essentially nothing more than a restatement 
of~\cite{HLP}*{Proposition~4}.  The statement of the algorithm is simplified by
the following notation:

Suppose $\alpha_1, \alpha_2, \alpha_3, \beta_1, \beta_2, \beta_3$ are elements 
of a field $\ell$ of characteristic not~$2$. Let $f$ and $g$ be the monic cubic
polynomials whose roots are the $\alpha_i$ and the $\beta_i$, respectively.  
Suppose further that $f$ and $g$ are separable and that the quantity
  $ \alpha_1(\beta_3 - \beta_2) 
  + \alpha_2(\beta_1 - \beta_3) 
  + \alpha_3(\beta_2 - \beta_1)$
is nonzero.
Set $\alpha_{ij} = \alpha_i - \alpha_j$ and $\beta_{ij} = \beta_i - \beta_j$, 
and define
\begin{align*}
a_1 &= \alpha_{32}^2/\beta_{32} + \alpha_{21}^2/\beta_{21} + \alpha_{13}^2/\beta_{13}, &
a_2 &= \alpha_1 \beta_{32} + \alpha_2 \beta_{13} + \alpha_3 \beta_{21}, \\
b_1 &= \beta_{32}^2/\alpha_{32} + \beta_{21}^2/\alpha_{21} + \beta_{13}^2/\alpha_{13}, &
b_2 &= \beta_1 \alpha_{32} + \beta_2 \alpha_{13} +  \beta_3\alpha_{21}.
\end{align*}
Let $A=\Delta_g a_1/a_2$ and $B=\Delta_f b_1/b_2$, where $\Delta_f$ and 
$\Delta_g$ are the discriminants of $f$ and $g$, respectively.  Then we let
$h_{\alpha_1,\alpha_2,\alpha_3,\beta_1,\beta_2,\beta_3}$ be the polynomial
\[
  - (A \alpha_{21} \alpha_{13} x^2 + B \beta_{21} \beta_{13})
    (A \alpha_{32} \alpha_{21} x^2 + B \beta_{32} \beta_{21})
    (A \alpha_{13} \alpha_{32} x^2 + B \beta_{13} \beta_{32}).
\]

\begin{algorithm}
\label{A:glue2}
\begin{algtop}
\algin  Weierstrass models of two elliptic curves $E_1$ and $E_2$ over a 
        computationally amenable field $k$ of characteristic not $2$.
\algout The set of genus-$2$ curves $C$ over $k$ such that there are
        degree-$2$ maps $\varphi_i\colon C\to E_i$ for $i=1$ and $i=2$ with 
        $\varphi_{2*}\varphi_1^* = 0$.
\end{algtop}
\begin{alglist}
\item   Initialize $L$ to be the empty list.
\item   \label{TWOstep1}
        Write $E_1$ and $E_2$ as $y^2 = f$ and $y^2 = g$, respectively, where 
        $f$ and $g$ are separable monic cubic polynomials in $k[x]$.  Let 
        $\Delta_f$ and $\Delta_g$ denote the discriminants of $f$ and~$g$.
\item   Compute the splitting fields of $f$ and $g$.  If the two fields are not
        isomorphic to one another as extensions of~$k$, output the empty set 
        and stop.  Otherwise, set $\ell$ to be the splitting field of $f$ 
        and~$g$.
\item   Compute the roots $\alpha_1, \alpha_2, \alpha_3$ of $f$ and 
        $\gamma_1, \gamma_2, \gamma_3$ of $g$ in~$\ell$.
\item   \label{TWOstep3}
        For every permutation $\sigma$ of the set $\{1,2,3\}$, do the 
        following:
        \begin{algsublist}
        \item Set $\beta_i = \gamma_{\sigma(i)}$ for each $i$.
        \item If the quantity
              $  \alpha_1(\beta_3 - \beta_2) 
               + \alpha_2(\beta_1 - \beta_3) 
               + \alpha_3(\beta_2 - \beta_1)$
              is nonzero, and if the map 
              $\psi\colon E_1[2](\ell) \to E_2[2](\ell)$ defined by 
              $(\alpha_i, 0)\mapsto(\beta_i,0)$ is Galois-equivariant,
              append the triple $(\beta_1,\beta_2,\beta_3)$ to $L$.
        \end{algsublist}
\item  Output the set of all curves 
       \[ y^2 = h_{\alpha_1,\alpha_2,\alpha_3,\beta_1,\beta_2,\beta_3} \]
       for all triples $(\beta_1,\beta_2,\beta_3)$ in $L$.
\end{alglist}
\end{algorithm}

\begin{theorem}
\label{T:glue2}
Algorithm~\textup{\ref{A:glue2}} runs in expected polynomial time
and produces correct output. The output list will be nonempty if and only if 
there is an isomorphism $E_1[2]\to E_2[2]$ of group schemes over $k$ that is 
not the restriction to $E_1[2]$ of a geometric isomorphism $E_1\to E_2$.
\end{theorem}

\begin{proof}
It is clear that the algorithm runs in expected polynomial time.  
To show that the output is correct, we must analyze the condition from 
Step~\ref{TWOstep3}(b) that the quantity 
$  \alpha_1(\beta_3 - \beta_2) 
 + \alpha_2(\beta_1 - \beta_3) 
 + \alpha_3(\beta_2 - \beta_1)$
be nonzero.  Note that  this quantity is equal to the determinant
\[
\left| \begin{matrix}
\alpha_1 & \beta_1 & 1 \\
\alpha_2 & \beta_2 & 1 \\
\alpha_3 & \beta_3 & 1 
\end{matrix} \right|,
\]
so it is nonzero if and only if there is no affine transformation taking the 
$\alpha_i$ to the~$\beta_i$, which is equivalent to the condition that the map 
$\psi\colon E_1[2](\ell)\to E_2[2](\ell)$ from Step~\ref{TWOstep3}(b) is not
the restriction to $E_1[2](\ell)$ of a geometric isomorphism $E_1 \to E_2$. 
Thus, in Step~\ref{TWOstep3}, the algorithm enumerates all isomorphisms
$E_1[2]\to E_2[2]$ of group schemes over $k$ that do not come from geometric
isomorphisms $E_1\to E_2$. The correctness of the output then follows
from~\cite{HLP}*{Propositions~3 and~4}.
\end{proof}

\begin{remark}
Suppose we write $E_1$ and $E_2$ in the form $y^2 = f$ and $y^2 = g$ for 
separable monic cubic polynomials $f$ and $g$ in~$k[x]$.  Since the 
characteristic of $k$ is not $2$, giving an isomorphism between the 
$2$-torsion group schemes $E_1[2]$ and $E_2[2]$ over $k$ is equivalent to 
giving a Galois-equivariant bijection between the points of order $2$ on $E_1$
and the points of order $2$ on $E_2$. To give such a bijection, one simply
needs to give a Galois-equivariant bijection between the roots of $f$ and the
roots of~$g$.  Such a bijection exists if and only if the splitting fields of 
$f$ and $g$ are isomorphic to one another as extensions of~$k$. When $k$ is 
finite, these splitting fields will be isomorphic to one another if and only if
$E_1$ and $E_2$ have the same number of $k$-rational points of order~$2$. Thus,
if Algorithm~\ref{A:glue2} is given two elliptic curves over a finite field
that have the same number of rational $2$-torsion points and that have 
different $j$-invariants, the output set will be nonempty.
\end{remark}

Next we give an algorithm for gluing two elliptic curves together along their
$3$-torsion subgroups.  As was the case for the preceding algorithm and its 
proof, it will be convenient to have some explicit formulas for the family of
genus-$2$ curves obtained by such $3$-gluings.  Such formulas (for part or all
of the family of such curves) have appeared in the literature, going back at
least to 1876 (see, for 
example,~\cites{Hermite,Goursat,Kuhn,Shaska}), but none of the
references we have found have all of the information we would like to have 
about this family of curves.  However, using these references, we were able to
work out all of the desired details; we have collected our results in 
Appendix~\ref{S:appendix}.

Some notation will be helpful:  If $k$ is a field, let $k^*$ act on the set of
quadruples $(a,b,c,d)\in k^4$ by setting
\[
  \lambda(a,b,c,d) = 
  (\lambda^2 a, \lambda^3 b, \lambda^{-2} c, \lambda^{-3} d)
\]
for $\lambda\in k^*$, and denote the orbit of $(a,b,c,d)$ under this action by 
$[a\col b\col c\col d]$.  We denote the set of these orbits by $P_k$.

\begin{algorithm}
\label{A:glue3}
\begin{algtop}
\algin  Weierstrass models of two elliptic curves $E_1$ and $E_2$ over a 
        computationally amenable field $k$ of characteristic neither $2$ 
        nor~$3$.
\algout The set of genus-$2$ curves $C$ over $k$ such that there are degree-$3$
        maps $\varphi_i\colon C\to E_i$ for $i=1$ and $i=2$ with 
        $\varphi_{2*}\varphi_1^* = 0$.
\end{algtop}
\begin{alglist}
\item   Initialize $L$ to be the empty list.
\item   \label{THREEequations}
        Let $j_1$ and $j_2$ be the $j$-invariants of $E_1$ and $E_2$, and 
        define elements of the polynomial ring $k[w,x,y,z]$ as follows:
        \begin{align*}
         g_1 &= 1728 (w^2 y + 4 w x z - 4 x^2 y^2)^3 - j_1 (w^3 + x^2)^2 (y^3 + z^2),\\
         g_2 &= 1728 (w y^2 + 4 x y z - 4 w^2 z^2)^3 - j_2 (w^3 + x^2) (y^3 + z^2)^2,\\
         g_3 &= 12 w y + 16 x z - 1.
        \end{align*}
\item   \label{THREEinvariants}
        Find all elements $[a\col b\col c\col d]$ of $P_k$ that satisfy $g_1$, 
        $g_2$, and $g_3$, and such that either $ab\ne 0$ or $cd\ne 0$.
\item   \label{THREEtwists}
        For every $p\in P_k$ from Step~\ref{THREEinvariants}:
        \begin{algsublist}
        \item Choose $a,b,c,d\in k$ such that $p = [a\col b\col c\col d]$.
        \item If $(a^3 + b^2)(c^3 + d^2) \neq 0$, compute representatives 
              $t\in k^*$ of the elements of the (finite and possibly empty) set
              $S\subseteq k^*/k^{*2}$ such that the curves $E_{a,b,c,d,t,1}$
              and $E_{a,b,c,d,t,2}$ from Appendix~\ref{S:appendix} are 
              isomorphic to $E_1$ and~$E_2$. For each such $t$, append the 
              quintuple $(a,b,c,d,t)$ to the list $L$. 
        \end{algsublist}
\item   \label{THREEstepj0}
        If $j_1 = j_2 = 0$, write $E_1$ in the form $v^2 = u^3 + e_1$ and $E_2$
        in the form $v^2 = u^3 + e_2$, with $e_1, e_2\in k$.  If 
        $e_1 e_2\in 4 k^{*6}$, set $b= e_1$ and $d = 1/(16 e_1)$, and append
        the quintuple $(0,b,0,d,2)$ to the list~$L$.
\item   \label{THREEstepj1728}
        If $j_1 = j_2 = 1728$, write $E_1$ in the form $v^2 = u^3 + e_1 u$ and 
        $E_2$ in the form $v^2 = u^3 + e_2 u$, with $e_1, e_2\in k$.  If 
        $e_1 e_2\in 108 k^{*4}$, set $a = e_1$ and $c = 1/(12 e_1)$, and 
        append the quintuple $(a,0,c,0,2)$ to the list $L$.
\item   If any quintuple in $L$ is equivalent to an earlier quintuple under the
        action of $k^*\times k^*$ described in Appendix~\ref{S:appendix},
        delete the later quintuple from~$L$.
\item   Output the set of all curves $C_{a,b,c,d,t}$ (from
        Appendix~\ref{S:appendix}) for $(a,b,c,d,t)\in L$.
\end{alglist}
\end{algorithm}

\begin{theorem}
\label{T:glue3}
Algorithm~\textup{\ref{A:glue3}} runs in expected polynomial time
and produces correct output. The output list will be nonempty if the following 
conditions are satisfied\textup{:} 
\begin{enumerate}
\item[(a)] there exists an isomorphism $E_1[3] \to E_2[3]$ of group schemes 
           that is an anti-isometry with respect to the Weil pairing\textup{;}
           and
\item[(b)] the curves $E_1$ and $E_2$ are not $2$-isogenous to one another over
           the algebraic closure of~$k$.
\end{enumerate}
\end{theorem}

\begin{proof}
If we show that Step~\ref{THREEinvariants} and Step~\ref{THREEtwists}(b) can be
completed in expected polynomial time, it will be clear that the
entire algorithm runs in expected polynomial time.

To begin, we note that an easy calculation shows that 
$[-1/4\col 1/8\col -1\col -1]$ is the only element 
$[a\col b\col c\col d]$ of $P_k$ that satisfies $g_1 = g_2 = g_3 = 0$ and for 
which $(a^3 + b^2)(c^3 + d^2) = 0$.  

Suppose $p = [a\col b\col c\col d]$ is an element of $P_k$ such that 
$g_1(p) = g_2(p) = g_3(p) = 0$ and such that $ab\ne 0$. Since $ab\ne 0$, there
is a unique representative for $p$ such that $a=b$.  With this normalization, 
we find that $g_1 = g_2 = g_3 = 0$ becomes a system of three equations in three
unknowns $a$, $c$, and~$d$.  Proposition~\ref{P:3covers} from the Appendix
shows that every solution to this system over $\kb$ with 
$(a^3 + a^2)(c^3 + d^2)\neq 0$ gives a genus-$2$ curve $C_{a,a,c,d,1}$ over 
$\kb$ along with degree-$3$ maps $\varphi_{a,a,c,d,1,1}$ and 
$\varphi_{a,a,c,d,1,2}$ to the elliptic curves over $\kb$ with $j$-invariants
equal to $j_1$ and $j_2$, and distinct solutions over $\kb$ give rise to
nonisomorphic triples $(C,\varphi_1,\varphi_2)$.  There are at most $24$ such 
triples~\cite{Kani}*{Theorem~1}, so there are at most $24$ solutions to the 
system over $\kb$ with $(a^3 + a^2)(c^3 + d^2)\neq 0$.  As we noted above,
there is only one solution with $(a^3 + a^2)(c^3 + d^2) = 0$.  Therefore, the
variety determined by $g_1 = g_2 = g_3 = 0$ and $a = b$ is $0$-dimensional, and
computing its points over $k$ is an expected polynomial-time 
computation.

Likewise, in expected polynomial time one can compute the points
$p = [a\col b\col c\col d]$ of $P_k$ such that $g_1(p) = g_2(p) = g_3(p) = 0$ 
and such that $cd\ne 0$.  Thus, Step~\ref{THREEinvariants} can be completed in 
expected polynomial time.

To show that Step~\ref{THREEtwists}(b) runs in expected polynomial
time, we note that over a field $k$ of characteristic neither $2$ nor $3$, it 
is easy to determine the set of $t\in k^*$ (modulo $k^{*2}$) such that the 
quadratic twist of one elliptic curve by $t$ is isomorphic to a second elliptic
curve: One simply writes the two curves in short Weierstrass form as
$y^2 = x^3 + Ax + B$ and $y^2 = x^3 + A'x + B'$, and computes the set of 
$t\in k^*$ such that $A' = At^2$ and $B' = Bt^3$. Finding these $t$ can clearly
be done in expected polynomial time. There is at most one solution 
$t$ to these equations, unless $A = A' = 0$ or $B = B' = 0$.  If $A = A' = 0$ 
the solutions, if any, all lie in the same class of $k^*/k^{*2}$.  If 
$B = B' = 0$ there are either $0$ or $2$ solutions; if there are $2$ solutions,
they lie in the same class of $k^*/k^{*2}$ if and only if $-1$ is a square 
in~$k$. Thus, Step~\ref{THREEtwists}(b), and hence the entire algorithm, runs 
in expected polynomial time.

Next we must show that the output of the algorithm is correct. We see from 
Proposition~\ref{P:3covers} that the set we intend the algorithm to output is 
equal to the set of all curves $C_{a,b,c,d,t}$ with $(a^3 + b^2)(c^3+d^2)t\ne0$ 
and $12ac + 16ad = 1$ such that for each~$i$, we have 
$E_{a,b,c,d,t,i} \cong E_i$.

Certainly every curve $C_{a,b,c,d,t}$ in the set output by the algorithm
satisfies $(a^3 + b^2)(c^3+d^2)t\ne0$ and $12ac + 16bd = 1$, and has the 
property that $E_{a,b,c,d,t,i} \cong E_i$ for each~$i$; 
Step~\ref{THREEtwists}(b) explicitly enforces these requirements, and an easy
calculation shows that the curves (if any) obtained from 
Steps~\ref{THREEstepj0} and~\ref{THREEstepj1728} also have these properties.

On the other hand, suppose $(a,b,c,d,t)$ is a quintuple such that 
$(a^3 + b^2)(c^3+d^2)t\ne0$ and $12ac + 16bd = 1$ and such that 
$E_{a,b,c,d,t,i} \cong E_i$ for each~$i$.  We see from equations \eqref{EQ:j1}
and \eqref{EQ:j2} that since $(a^3 + b^2)(c^3+d^2)t\ne0$, the equations $g_1$ 
and $g_2$ in Step~\ref{THREEequations} express the condition that the elliptic 
curves $E_{a,b,c,d,t,i}$ and $E_i$ have the same $j$-invariant, for $i=1$ and
$i=2$.  Thus, Steps~\ref{THREEinvariants} and~\ref{THREEtwists} ensure that
the algorithm will find $(a,b,c,d,t)$ (up to the action of $k^*\times k^*$) if
$ab\ne 0$ or $cd\ne 0$.

Suppose our quintuple $(a,b,c,d,t)$ has $ab=0$ and $cd=0$.  We see from the
condition that $12ac + 16bd = 1$ that then either $a=c=0$ or $b=d=0$.  If
$a=c=0$, then equations \eqref{EQ:j1} and \eqref{EQ:j2} show that $j_1=j_2=0$,
and we find that $E_1$ and $E_2$ must be isomorphic to the curves
\[
t y^2 = x^3 + 512 b^4 d^3 \quad\text{and}\quad
t y^2 = x^3 + 512 b^3 d^4,
\]
respectively.  Using the condition that $16bd = 1$ and rescaling the variables 
$x$ and $y$, we find that $E_1$ and $E_2$ are isomorphic to 
\[
y^2 = x^3 +  b t^3 / 8 \quad\text{and}\quad
y^2 = x^3 +  d t^3 / 8.
\]
Thus, the $e_1$ and $e_2$ from Step~\ref{THREEstepj0} must satisfy 
$e_1 = b t^3 r^6 / 8$ and $e_2 = d t^3 s^6/8$ for some $r,s\in k^*$. It follows
that $e_1 e_2\in 4 k^{*6}$.  Furthermore, given any $e_1$ and $e_2$ with 
$e_1 e_2\in 4 k^{*6}$, if we take $b = e_1$, $d = 1/(16e_1)$, and $t=2$, then 
$E_{0,b,0,d,t,i} \cong E_i$ for each~$i$.

On the other hand, if $b=d=0$, then equations \eqref{EQ:j1} and \eqref{EQ:j2} 
show that $j_1=j_2=1728$, and we find that $E_1$ and $E_2$ must be isomorphic
to the curves
\[
t y^2 = x^3 + 36 a^3 c^2 \quad\text{and}\quad
t y^2 = x^3 + 36 a^2 c^3,
\]
respectively.  Using the condition that $12ac = 1$ and rescaling the variables 
$x$ and $y$, we find that $E_1$ and $E_2$ are isomorphic to 
\[
y^2 = x^3 +  (a t^2 / 4)x\quad\text{and}\quad
y^2 = x^3 +  (c t^2 / 4)x.
\]
Thus, the $e_1$ and $e_2$ from Step~\ref{THREEstepj1728} must satisfy 
$e_1 = a t^2 r^4 / 4$ and $e_2 = c t^2 s^4 / 4$ for some $r,s\in k^*$.  It 
follows that $e_1 e_2\in 108 k^{*4}$.  Furthermore, given any $e_1$ and $e_2$
with $e_1 e_2\in 108 k^{*4}$, if we take $a = e_1$, $c = 1/(12e_1)$, and $t=2$,
then $E_{a,0,c,0,t,i} \cong E_i$ for each~$i$.

Thus, every curve $C$ in the list output by the algorithm does have degree-$3$
maps $\varphi_i\colon C\to E_i$ such that $\varphi_{2*}\varphi_1^* = 0$, so the
output is correct.

Now suppose conditions (a) and (b) of the theorem hold. Condition (a) says that
there is an isomorphism $\psi\colon E_1[3] \to E_2[3]$ of group schemes that is
an anti-isometry with respect to the Weil pairing.  As is explained 
in~\cite{FK}, associated to this data there is a (possibly singular) curve $C$
of arithmetic genus $2$, together with degree-$3$ maps 
$\varphi_1\colon C\to E_1$ and $\varphi_2\colon C\to E_2$ such that 
$\varphi_{2*}\varphi_1^* = 0$.  Combining condition (b) with a result of 
Kani~\cite{Kani}*{Theorem~3} we find that $C$ is in fact a nonsingular curve.
Thus $C$ will appear in the list output by the algorithm, so the list is
nonempty.
\end{proof}

\section{Genus-2 curves of given order}
\label{S:curves}

\noindent
In this section we prove Theorem~\ref{Th:C}.  Our strategy will be to look at
curves $C$ over finite prime fields $\FF_p$ such that the Jacobian $J$ of $C$
is isogenous to a product $E_1\times E_2$ of elliptic curves. As we noted in 
Equation~\eqref{EQ:traces} in Section~\ref{S:genus2}, if $E_1$ and $E_2$ have
traces $t_1$ and $t_2$, then $C$ will have $p + 1 - t_1 - t_2$ rational
points. Leaving aside for the moment the question of how to produce $C$ from
$E_1$ and $E_2$, we see that if we are given an integer $N$, we would like to 
produce a prime $p$ and two elliptic curves $E_1$ and $E_2$ over $\FF_p$ with
traces that sum to $p + 1 - N$.

In Section~\ref{S:intro} we noted the difficulty constructing an elliptic curve
over a given finite field with a given trace of Frobenius (Problem~A). However, 
there is an easy special case of this problem: Given a prime $p$, it is very 
easy to produce a supersingular elliptic curve over $\FF_p$ 
(see~\cite{Broker}), and for $p>3$ all such curves have trace~$0$.  We 
therefore use the following strategy for producing a curve $C$ with a given
number $N$ of points:
\begin{itemize}
\item Construct an elliptic curve $E_1$ over some prime field $\FF_p$ such that
      the trace $t_1$ of $E_1$ satisfies $t_1 = p + 1 - N$; that is, 
      $\#E_1(\FF_p) = N$.
\item Construct a supersingular curve $E_2$ over $\FF_p$, so that the trace 
      $t_2$ of $E_2$ satisfies $t_2 = 0$.
\item Construct a genus-$2$ curve $C$ over $\FF_p$ whose Jacobian is isogenous
      to $E_1\times E_2$.  Then from Equation~\eqref{EQ:traces} we find
      \[\#C(\FF_p) = p + 1 - t_1 - t_2 = N.\]
\end{itemize}

In order to obtain an actual algorithm from this outline, we begin with some 
results that help us produce elliptic curves to use as input data for
Algorithms~\ref{A:glue2} and~\ref{A:glue3}. In order to obtain one or more 
genus-$2$ curves from one of these algorithms, the two elliptic curves that are
input to the algorithm must have isomorphic $\ell$-torsion subgroup schemes, 
where $\ell=2$ for Algorithm~\ref{A:glue2} and $\ell=3$ for 
Algorithm~\ref{A:glue3}. Definition~\ref{D:minimal} (below) and the results
that follow it help us produce elliptic curves in a given isogeny class whose 
$\ell$-torsion subgroup schemes have a known structure.

Let $E$ be an elliptic curve over a finite field $k$ of cardinality $q$, and 
let $\pi$ denote the Frobenius endomorphism of $E$. Suppose that the
endomorphism ring of $E$ is an order in an imaginary quadratic field.  (This 
will be the case precisely when the endomorphism ring is commutative, and 
precisely when $\pi$ does not lie in~$\ZZ$.) Then the ring $\ZZ[\pi]$ is a 
subring of finite index in $\End E$.

\begin{definition}
\label{D:minimal}
Let $\ell$ be a prime.  The elliptic curve $E$ is \emph{minimal at $\ell$} if 
the index of $\ZZ[\pi]$ in $\End E$ is not divisible by $\ell$.
\end{definition}

Let $\Delta$ be the discriminant of $\End E$ and let $t$ be the trace of the 
Frobenius endomorphism of $E$, so that the discriminant of $\ZZ[\pi]$ is equal
to $t^2 - 4q$.  We see that $E$ is minimal at $\ell$ if and only if 
$(t^2 - 4q)/\Delta$ is not divisible by~$\ell$.

\begin{lemma}
\label{L:minimal}
Let $E$ be an elliptic curve over a finite field~$k$ whose endomorphism ring is
commutative, and let $\ell$ be a prime not equal to the characteristic of~$k$. 
Then $E$ is minimal at $\ell$ if and only if the number of $k$-rational 
rank-$\ell$ subgroup schemes of $E$ is less than $\ell+1$.
\end{lemma}

\begin{proof}
Let $V$ be the group $E[\ell](\kb)$, viewed as a $2$-dimensional 
$\FF_\ell$-vector space.  The Frobenius endomorphism $\pi$ of $E$ acts 
invertibly on $V$, so we can view it as an element $x$ of $\GL(V)$.  The
rank-$\ell$ subgroup schemes of $E$ correspond to $1$-dimensional eigenspaces
of~$x$, so there will be $\ell+1$ of these subgroup schemes when $x$ acts as a
scalar, and fewer than $\ell+1$ subgroup schemes otherwise. If $x$ acts as 
multiplication by an integer $a$, then $\pi-a$ kills all of $E[\ell]$, so the 
endomorphism $\pi-a$ of $E$ factors through multiplication by $\ell$, and 
$(\pi-a)/\ell$ is an endomorphism $\alpha$ of~$E$.  Conversely, if 
$(\pi-a)/\ell$ is an endomorphism of $E$, then $\pi$ acts as a scalar on 
$E[\ell]$.

Thus, there are $\ell+1$ $k$-rational rank-$\ell$ subgroup schemes of $E$ 
precisely when $\pi$ lies in $\ZZ + \ell\End E$, which is the case precisely 
when the index of $\ZZ[\pi]$ in $\End E$ is divisible by $\ell$.

(In the ordinary case, the lemma also follows 
from~\cite{FM}*{Theorem~2.1, p.~278}.)
\end{proof}

The next algorithm shows that it is easy to produce curves that are minimal at
a given prime.

\begin{algorithm}
\label{A:minimal}
\begin{algtop}
\algin  A triple $(E, H, \ell)$, where $E$ is an ordinary elliptic curve over a 
        finite field $k$ of characteristic greater than $3$ such that $\End E$ 
        is a maximal order in a quadratic field, where $H$ is the image in 
        $k[x]$ of the Hilbert class polynomial of this maximal order, and where 
        $\ell$ is an integer equal to either $2$ or $3$.
\algout An elliptic curve over $k$ that is isogenous to $E$ and that is minimal
        at~$\ell$.
\end{algtop}
\begin{alglist}
\item   If $E$ has fewer than $\ell+1$ subgroup schemes of rank~$\ell$, return
        $E$ and stop.
\item   Choose a rank-$\ell$ subgroup scheme $G$ of $E$ so that the
        $j$-invariant of the quotient curve $E/G$ is not a root of $H$.
\item   Set $E_0 = E$ and $E_1 = E/G$, and set $i = 1$.
\item   \label{MINIMALstep4}
        If $E_i$ has fewer than $\ell+1$ rank-$\ell$ subgroup schemes, return 
        $E_i$, and stop.
\item   Pick a rank-$\ell$ subgroup scheme $G_i$ of $E_i$ such that $E_i/G_i$ 
        is not isomorphic to $E_{i-1}$, and set $E_{i+1} = E_i / G_i$.
\item   Increment $i$, and go to Step~\ref{MINIMALstep4}.
\end{alglist}
\end{algorithm}

\begin{remark}
If we write an elliptic curve $E/k$ as $y^2 = x^3 + ax + b$, then the rank-$2$
subgroup schemes of $E$ correspond to the roots of $x^3 + ax + b$ in $k$; a
root $r$ corresponds to the rank-$2$ subgroup scheme $G$ that contains the 
point $(r,0)$, and the quotient $E/G$ can be written 
$y^2 = x^3 - (4 a + 15 r^2) x + (14 a r + 22 b)$.  The rank-$3$ subgroup 
schemes of $E$ correspond to the roots of the $3$-division polynomial 
$3x^4 + 6ax^2 + 12bx - a^2$; a root $r$ corresponds to the rank-$3$ subgroup
scheme $G$ that contains the two geometric points of $E$ with $x$-coordinate
equal to $r$, and the quotient $E/G$ can be written 
$x^3 - (9 a + 30 r^2) x - (42 a r + 27 b + 70 r^3)$.
\end{remark}

\begin{theorem}
Algorithm~\textup{\ref{A:minimal}} is correct, and runs in expected
polynomial time.
\end{theorem}

\begin{proof}
The algorithm follows a path, without backtracking, along the `isogeny volcano'
of $\ell$-isogenies~\cite{FM} (see also~\cite{Kohel}*{\S4.2}).  The curve $E_0$ 
is on the rim of the volcano, and the condition that the $j$-invariant of $E/G$
not be a root of $H$ ensures that $E_1$ is not on the rim of the volcano.
Therefore the isogeny $E_0\to E_1$ is `descending,' and the general theory 
shows all of the successive isogenies in the path are also descending.   The 
maximal number of steps on the descending path before an $\ell$-minimal curve
is reached is the $\ell$-adic valuation of the conductor of the order of 
discriminant $t^2 - 4q$, which is polynomial in the input size.
\end{proof}

\begin{remark}
In Algorithm~\ref{A:minimal}, we restrict $\ell$ to be $2$ or $3$ merely to 
avoid a discussion on the representation of subgroup schemes of larger rank.
\end{remark}

\begin{remark}
In general, one can easily produce an $\ell$-minimal curve isogenous to a
given~$E$, even when $\End E$ is not maximal and when no Hilbert class
polynomial is provided; one simply traverses three paths starting at $E$, but
with different first steps.  One of the paths is guaranteed to be descending. 
However, in our application we will have the Hilbert class polynomial at hand
anyway, so we give this slightly simpler algorithm.
\end{remark}

Now we reach the algorithm that we will use to prove Theorem~\ref{Th:C}.

\begin{algorithm}
\label{A:curve}
\begin{algtop}
\algin   A positive integer $N\not\equiv 1\bmod 6$ together with its 
         factorization.
\algout  A prime $p$ and a genus-$2$ curve $C$ over $\FF_p$ such that 
         $\#C(\FF_p) = N$, or the word `Failed'.
\end{algtop}
\begin{alglist}
\item    \label{CURVE-setell}
         If $N$ is even, set $\ell = 2$.  Otherwise, set $\ell = 3$.
\item    \label{CURVE-BS}
         Use the modified version of the algorithm of Br\"oker and 
         Stevenhagen~\cite{BS} discussed below in Remark~\ref{R:BS} to try to
         produce a fundamental discriminant~$\Delta$, the Hilbert class
         polynomial $H$ for~$\Delta$, a prime~$p>3$ congruent to $N-1$ 
         modulo~$\ell$, and an ordinary elliptic curve $E$ over $\FF_p$ with CM
         by $\Delta$ and with $\#E(\FF_p) = N$. If this step fails, output 
         `Failed' and stop.
\item    \label{CURVE-E1}
         Apply Algorithm~\ref{A:minimal} to $E$, $H$, and $\ell$ to find an 
         elliptic curve $E_1$ over $\FF_p$, isogenous to $E$, that is minimal 
         at~$\ell$.
\item    \label{CURVE-E2}
         Use the algorithm of Br\"oker~\cite{Broker} to produce a trace-$0$ 
         elliptic curve $E_2$ over~$\FF_p$.
\item    \label{CURVE-2}
         If $\ell = 2$ do the following:
         \begin{algsublist}
         \item If $E_2$ has three rational points of order~$2$, replace $E_2$
               by a $2$-isogenous curve that has only one rational point of 
               order~$2$.
         \item Apply Algorithm~\ref{A:glue2} to $E_1$ and $E_2$, choose a curve
               $C$ from the resulting list, output $p$ and $C$, and stop.
         \end{algsublist}
\item    \label{CURVE-3}
         If $\ell=3$ do the following:
         \begin{algsublist}
         \item Apply Algorithm~\ref{A:glue3} to the curves $E_1$ and $E_2$.
               If the algorithm returns a nonempty list of curves, choose a 
               curve $C$ from the list, output $p$ and $C$, and stop.
         \item Compute a curve $E_2'$ that is $2$-isogenous to~$E_2$.
         \item Apply Algorithm~\ref{A:glue3} to the curves $E_1$ and $E_2'$,
               choose a curve $C$ from the list returned by the algorithm,
               output $p$ and $C$, and stop.
         \end{algsublist}
\end{alglist}
\end{algorithm}

\begin{remark}
\label{R:BS}
Recall the outline of the Br\"oker--Stevenhagen algorithm~\cite{BS}*{p.~2168},
sketched in Section~\ref{S:genus1}: Given a positive integer $N$, together with
its factorization, the algorithm will produce a pair $(d,\nu)$, where $d$ is a
squarefree positive integer and $\nu$ is an integer of the field 
$\QQ(\sqrt{-d})$ such that $\nu$ has norm $N$ and $1-\nu$ has norm equal to a
prime. The algorithm runs by looking at each imaginary quadratic field $K$ in
turn, finding all integers $\nu\in K$ of norm $N$, and waiting until one of 
these $\nu$ satisfies the condition that the norm of $1-\nu$ is prime.

For Step~\ref{CURVE-BS} of Algorithm~\ref{A:curve}, we need to use a version of
the Br\"oker--Stevenhagen algorithm modified as follows: The input to the 
algorithm now includes an auxiliary prime $\ell$. As in the original algorithm,
we run through fields $K$ and integers $\nu$ of $K$ with norm $N$ until we find 
an $\nu$ such that the norm of $1-\nu$ is a prime $p$, but now we add in three
additional restrictions:
\begin{enumerate}
\item $p > 3$, 
\item $p \neq N-1$, and 
\item $p \equiv N-1\bmod \ell$.
\end{enumerate}
(Given the first condition, the second condition is equivalent to requiring the
algorithm to output an ordinary elliptic curve.) As we explained in the proof
of Lemma~\ref{L:BS}, this modified algorithm has a heuristic expected running
time polynomial in $\ell 2^{\omega(N)} \log N$.
\end{remark}

\begin{proof}[Proof of Theorem~\textup{\ref{Th:C}}]
We will prove Theorem~\ref{Th:C} by showing that Algorithm~\ref{A:curve} has 
the required properties.

Note that if there exist a prime $p$ with $p\equiv N-1\bmod \ell$ and an
ordinary elliptic curve $E$ over $\FF_p$ with $\#E(\FF_p)=N$, then there also
exists such a $p$ and $E$ with $p > 3$.  This is easy to check when $N<8$, and
when $N\ge 8$ the condition $p > 3$ follows from the condition that
$\#E(\FF_p)=N$.

As we discussed in Remark~\ref{R:BS}, the modified Br\"oker--Stevenhagen
algorithm will succeed in producing an ordinary elliptic curve $E$ over a prime 
field $\FF_p$ with $\#E(\FF_p)=N$ and with $p\equiv N-1\bmod \ell$, whenever 
such a curve exists.  Under standard heuristic assumptions, the modified 
algorithm runs in time polynomial in $\ell 2^{\omega(N)} \log N$, and since 
$\ell\le 3$ this is also polynomial in $2^{\omega(N)} \log N$. Thus, for the
rest of the proof, we may assume that Step~\ref{CURVE-BS} succeeds in producing
an $E$ over a prime field $\FF_p$ as above --- and, as we noted, we may also
assume $p > 3$. Since $E$ is ordinary, it cannot have trace~$0$.

Step~\ref{CURVE-E1} will run in expected polynomial time.

Br\"oker's algorithm will produce a supersingular curve $E_2$ over $\FF_p$, 
and will run in polynomial time if the Generalized Riemann Hypothesis is true.
Thus Step~\ref{CURVE-E2} will succeed in polynomial time under standard 
hypotheses.

Suppose $N$ is even. Let $\pi$ denote the Frobenius of $E_2$, so that 
$\pi^2 + p = 0$ and the index of the ring $\ZZ[\pi]$ in $\End E_2$ is either
$1$ or~$2$. It follows that every curve isogenous to $E_2$ that is not minimal
at~$2$ is $2$-isogenous to one that is. Since a curve of even order is minimal
at~$2$ if and only if it has just one rational point of order~$2$, we see that
Step~\ref{CURVE-2}(a) will succeed.

The curve $E_1$ from Step~\ref{CURVE-E1} is minimal at $2$ and has an even 
number $N$ of points, so it also has exactly one rational point of order~$2$.
Since $E_1$ and $E_2$ are defined over a finite field, it follows that the 
group schemes $E_1[2]$ and $E_2[2]$ are isomorphic to one another.  Also, since
$E_1$ is ordinary and $E_2$ is supersingular, the two curves have different 
$j$-invariants. Thus, in Step~\ref{CURVE-2}(b), Algorithm~\ref{A:glue2} will 
succeed in producing a genus-$2$ curve $C$ whose Jacobian is isogenous to
$E_1\times E_2$, so that $C$ will have $N$ points.

Finally, suppose we have reached Step~\ref{CURVE-3}, and suppose that in 
Step~\ref{CURVE-3}(a), Algorithm~\ref{A:glue3} fails to return a curve~$C$. 
According to Theorem~\ref{T:glue3}, this can only happen if there is \emph{not}
an anti-isometry $E_1[3]\to E_2[3]$, or if there \emph{is} a $2$-isogeny from 
$E_1$ to $E_2$ over the algebraic closure of the base field. However, since 
$E_1$ is ordinary and $E_2$ is supersingular, the two curves are not 
geometrically isogenous to one another, so there is no anti-isometry 
$E_1[3]\to E_2[3]$.

The curve $E_2$ has even order, so we can compute a $2$-isogenous curve $E_2'$ 
as required by Step~\ref{CURVE-3}(b). Recall that $E_1$ was constructed to be 
minimal at $3$, and note that $E_2$ is also minimal at $3$, because (as we 
noted earlier) the index of the ring $\ZZ[\pi]$ in $\End E_2$ is either $1$
or~$2$. Therefore we can apply \cite{HNR}*{Lemma~4.3, p.~249}, and we find that
either there is an anti-isometry $E_1[3]\to E_2[3]$ or there is an 
anti-isometry $E_1[3]\to E_2'[3]$.  Since there is not one from $E_1[3]$ to
$E_2[3]$, there must be one from $E_1[3]$ to $E_2'[3]$.  Combining this with
the fact that $E_1$ and $E_2'$ are not geometrically isogenous (because one
curve is ordinary and the other supersingular) and applying 
Theorem~\ref{T:glue3}, we find that Algorithm~\ref{A:glue3} applied to $E_1$
and $E_2'$ will produce at least one curve $C$.  Therefore, 
Step~\ref{CURVE-3}(c) will succeed.
\end{proof}

\section{Explicit examples}
\label{S:examples}

\noindent
We conclude by explicitly constructing several genus-2 curves having a 
prescribed large number of points.  The large numbers we chose for our examples
are $N_1=10^{2013}$ and $N_2 = 10^{2014} + 9703$, the smallest prime larger
than $10^{2014}$.  One of our examples we are able to specify completely here; 
the equations for the others can be found on the second author's web site, by 
starting at 

\centerline{\href{http://alumni.caltech.edu/~however/biblio.html}
                {\texttt{http://alumni.caltech.edu/{\lowtilde}however/biblio.html}}
}
\noindent
and following the link associated with this paper.  

\subsection*{A genus-2 curve of order \texorpdfstring{$10^{2013}$}{10 to the 2013}}

The first step in our construction is to produce an elliptic curve of order
$N_1=10^{2013}$.  As explained in \cite{BS}, elliptic curves of $10$-power 
order can often be constructed with endomorphism ring $\ZZ[i]$, the smallest 
imaginary quadratic order in which either $2$ or $5$ splits completely; the 
order $\ZZ[(-1+\sqrt{-31})/2]$ of discriminant $-31$, in which both $2$ and $5$
split completely, is expected to work in all cases.  We will show that both of
these orders can be used to produce a curve of order~$N_1$.

Let $i$ be a square root of $-1$, and take
\[
\nu = 2^{1006} \cdot 5^{164} \cdot (1 + i) \cdot (2 + i)^{1685}.
\]
Then $\Norm(\nu) = 10^{2013}$ and $p = \Norm(1-\nu)$ is prime, and the elliptic
curve $E\colon y^2 = x^3 - x$ over $\FF_p$ has $10^{2013}$ points. However, 
this curve is not minimal at~$2$; in fact, the large power of $2$ that appears 
in $\nu$ ensures that the index of $\ZZ[\pi]$ in $\End E$ is divisible by 
$2^{1006}$. Therefore, to find an isogenous curve $E_1$ that is minimal at~$2$,
we must travel $1006$ steps down a very tall isogeny volcano. This can be done 
without much trouble, but there is no clear way of expressing the $j$-invariant
of the resulting curve in a compact manner.

The prime $p$ is inert in the quadratic field of discriminant $-19$, so any 
curve over $\FF_p$ with CM by the order of discriminant $-19$ must be 
supersingular (and have trace $0$).  The Hilbert class polynomial for this 
discriminant is $x + 96^3$, so we can take $E_2$ to be any curve over $\FF_p$
with $j$-invariant $-96^3$.  The discriminant of the characteristic polynomial
of Frobenius for $E_2$ is~$-4p$, which is a fundamental discriminant because 
$p\equiv 1\bmod 4$.  It follows that $E_2$ is minimal at~$2$. Gluing $E_1$ and 
$E_2$ together along their $2$-torsion subgroups gives us a genus-$2$ curve $C$
over $\FF_p$ with $10^{2013}$ points.   

We chose our $E_2$ so that the curve $C$ that we obtained could be written
as $y^2 = x^6 + c_4 x^4 + c_2 x^2 + 1$, for certain $c_2,c_4$ in $K$. This 
curve has obvious maps to the elliptic curves
$y^2 =  x^3 + c_4 x^2 + c_2 x + 1$ and $y^2 = x^3 + c_2 x^2 + c_4 x + 1$.
At the URL mentioned above, we give the values of $c_2$ and $c_4$, as well as 
Magma code that shows that the two quotient elliptic curves have the number of
points that we claim.

\subsection*{Another genus-2 curve of order \texorpdfstring{$10^{2013}$}{10 to the 2013}}

To avoid the long chain of $2$-isogenies that the preceding construction 
required, we can replace the order $\ZZ[i]$ with an order in which $2$ splits,
and then require that $\nu$ not be divisible by many powers of~$2$.  (We will
have to take $\nu$ to be divisible by~$2$, in order for $1-\nu$ to have prime 
norm.) For this example, we use the order $\ZZ[\omega]$ of discriminant~$-31$,
where $\omega=(-1+\sqrt{-31})/2$. We find that the integer
\[
\nu =  2 (\omega - 1) \cdot 5^{322} \cdot (4\omega + 1)^{456} (\omega + 1)^{670}
\]
has norm $10^{2013}$, and $p = \Norm(1-\nu)$ is a $2014$-digit prime. If we 
then take $E$ to be the appropriate twist of an elliptic curve over $\FF_p$ 
whose $j$-invariant is a root of the Hilbert class polynomial for 
discriminant $-31$, we will have $\#E(\FF_p)=10^{2013}$.  For this~$E$, we need
take only one step down the isogeny volcano to find an isogenous curve $E_1$ 
that is minimal at~$2$.  Since $p\equiv 3 \bmod 4$, we can take $E_2$ to be the
curve $y^2 = x^3 + x$.  Gluing $E_1$ and $E_2$ together along their $2$-torsion
subgroups gives us a genus-$2$ curve $C$ over $\FF_p$ with $10^{2013}$ points.
Carrying out this procedure and cleaning up the resulting equations as much as
possible, we obtain the following result.

\begin{theorem}
\label{T:2014}
Let $p$ be the $2014$-digit prime specified in the preceding paragraph, and
let $u$ be any one of the three elements of $\FF_p$ that satisfies 
$u^3 + u + 1 = 0.$  Then the genus-$2$ curve $C/\FF_p$ defined by
\[
y^2 = (u - 1) (x^2 + 8) (x^4 + 16 x^2 + u^{24})
\]
has exactly $10^{2013}$ rational points.
\end{theorem}

Magma code verifying this example can be found at the URL mentioned above.

\subsection*{A genus-2 curve of order \texorpdfstring{$10^{2014}+9703$}{10 to the 2014 plus 9703}}

Again, to produce a genus-$2$ curve of order $N_2 = 10^{2014}+9703$, our
algorithm requires that we start with an elliptic curve of order~$N_2$. The 
Br\"oker--Stevenhagen algorithm produced an elliptic curve $E_1$ over a field 
$\FF_p$, with $\End E_1$ the quadratic order of discriminant
$-96097\cdot 127363$.  Producing the curve $E_1$ required finding a root in
$\FF_p$ of a class polynomial for this quadratic order; we thank Andrew
Sutherland for carrying out the computation for us, using the methods 
of~\cite{Sutherland}.

Since $N_2$ is odd, we must take $\ell= 3$ in Algorithm~\ref{A:curve}.  We 
compute that the curve $E_1$ is minimal at $\ell$.

The prime $p$ is congruent to $-1$ modulo $7$, so $p$ is inert in the
quadratic field $\QQ(\sqrt{-7})$, and hence the elliptic curve $E_2$ over 
$\FF_p$ defined by $y^2 = x^3 - 35 x + 98$, which has CM by the order of 
discriminant $-7$, is supersingular and has trace~$0$. Applying 
Algorithm~\ref{A:glue3}, we find a genus-$2$ curve $C$ with degree-$3$ maps to
both $E_1$ and $E_2$, and this $C$ therefore has exactly $N_2$ rational points.
Equations for $E_1$ and $C$ can be found at the URL mentioned above.

Note that even though the input $N_2$ is a number that we did not \emph{prove}
to be prime, the output of our algorithm is correct if the input is; that is,
if $N_2$ is indeed prime.  Actually, the fact that Algorithm~\ref{A:curve}  
produces any output at all is already a strong probabilistic proof of the
primality of $N_2$, because the Br\"oker--Stevenhagen subroutine in 
Step~\ref{CURVE-BS} requires the computation of a large number of square roots
of potential discriminants $\Delta$ modulo $N_2$ in order to succeed.

\appendix
\section{Genus-2 triple covers of elliptic curves}
\label{S:appendix}

\noindent
As we noted in Section~\ref{S:gluing}, explicit families of genus-$2$ curves 
with degree-$3$ maps to elliptic curves appeared in the literature over $125$ 
years ago.  Indeed, in addition the family of curves given by 
Goursat~\cite{Goursat} in 1885, which includes every genus-$2$ curve over
$\CC$ with a degree-$3$ map to an elliptic curve with a single exception, there
is also an 1876 paper of Hermite~\cite{Hermite} that gives formulas for 
the $1$-parameter family of triple covers $C\to E$ over $\CC$ called `special'
by Kuhn~\cite{Kuhn} and `degenerate' by Shaska~\cite{Shaska}, and that includes
the curve missed by Goursat's family.

However, neither these $19$th century works nor their modern counterparts
provide exactly what we would like to have: a complete parametrization, over an
arbitrary base field $k$, of the family of genus-$2$ curves over $k$ that have 
$k$-rational degree-$3$ maps to elliptic curves, including formulas for the 
genus-$2$ curves, the associated elliptic curves, and the degree-$3$ maps.  In 
this appendix we provide such parameterizations, the sole restriction being 
that we assume the characteristic of $k$ is neither $2$ nor~$3$.  The family of
genus-$2$ curves we obtain is essentially identical to that of 
Goursat~\cite{Goursat}*{Exemple~II, pp.~155--157}.

\subsection{The parameterization}
We start by writing down a family of curves and maps.  Let $k$ be a field of 
characteristic neither $2$ nor $3$, and let $a,b,c,d,t$ be elements of $k$
satisfying
\begin{equation}
\label{EQ:conditions}
12 a c + 16 b d = 1, \qquad a^3 + b^2 \ne 0,
                     \qquad c^3 + d^2 \ne 0,
                     \qquad t \ne 0.
\end{equation}
Set $\Delta_1 = a^3 + b^2$ and $\Delta_2 = c^3 + d^2$, and define polynomials
$f, f_1, f_2$ by
\begin{align*}
f_{\ }   &= (x^3 + 3 a x + 2 b) (2 d x^3 + 3 c x^2 + 1),\\
f_1 &= x^3 + 12 ( 2 a^2 d - b c) x^2
           + 12 (16 a d^2 + 3 c^2) \Delta_1 x + 512 \Delta_1^2 d^3, \\
f_2 &= x^3 + 12 ( 2 b c^2 - a d) x^2
           + 12 (16 b^2 c \,+ 3 a^2) \Delta_2 x + 512 \Delta_2^2 b^3. 
\end{align*}
Further, define rational functions $u_1, v_1, u_2, v_2$ by
\begin{align*}
u_1 &= 12 \Delta_1 \, \frac{ - 2 d x + c }             {x^3 + 3 a x + 2 b},   &
v_1 &=    \Delta_1 \, \frac{ 16 d x^3 - 12 c x^2 - 1 }{(x^3 + 3 a x + 2 b)^2},\\
u_2 &= 12 \Delta_2 \, \frac{ x^2 (a x - 2 b)}      {2 d x^3 + 3 c x^2 + 1}, &
v_2 &=    \Delta_2 \, \frac{ x^3 + 12 a x - 16 b }{(2 d x^3 + 3 c x^2 + 1)^2}.
\end{align*}

The following lemma is purely computational, and we leave the proof to the 
reader and his or her computational algebra package.

\begin{lemma}
\label{L:maps}
The discriminants of $f$, $f_1$, and $f_2$ are
\begin{align*}
\Delta(f)_{\ } &= \phantom{-} 2^8 \, 3^{12}  \Delta_1^3    \Delta_2^3, \\
\Delta(f_1)    &=          -  2^2 \, 3^{3\ } \Delta_1^2    \Delta_2,   \\
\Delta(f_2)    &=          -  2^2 \, 3^{3\ } \Delta_1^{\ } \Delta_2^2.
\end{align*}
Furthermore, for each $i = 1,2$ there is a degree-$3$ morphism from the curve
$t\, y^2 = f$ to the curve $t\, y^2 = f_i$ given by sending $(x,y)$ to 
$(u_i, y v_i)$.
\end{lemma}

The lemma shows that given $a,b,c,d,t$ in $k$ that 
satisfy~\eqref{EQ:conditions}, we obtain a genus-$2$ curve $C_{a,b,c,d,t}$
defined by $t\, y^2 = f$, two elliptic curves $E_{a,b,c,d,t,1}$ and 
$E_{a,b,c,d,t,2}$ defined by $t\, y^2 = f_1$ and $t\, y^2 = f_2$, and 
degree-$3$ maps
\begin{align*}
\varphi_{a,b,c,d,t,1}\colon C_{a,b,c,d,t}&\to E_{a,b,c,d,t,1} &
\varphi_{a,b,c,d,t,2}\colon C_{a,b,c,d,t}&\to E_{a,b,c,d,t,2} \\
               (x,y) &\mapsto (u_1, y v_1) &
               (x,y) &\mapsto (u_2, y v_2) 
\end{align*}
It is easy to choose values of $a,b,c,d,$ and $t$ in $\QQ$ so that the 
curves $E_{a,b,c,d,t,1}$ and $E_{a,b,c,d,t,2}$ are geometrically nonisogenous.
Therefore, for generic values of $a,b,c,d,$ and $t$ the morphism
\[
\varphi_{a,b,c,d,t,2*}\,\varphi_{a,b,c,d,t,1}^*
\colon E_{a,b,c,d,t,1}\to E_{a,b,c,d,t,2}
\]
is the zero map, so it must be the zero map for all values of $a,b,c,d,$ and 
$t$ in any field.  It follows that $\Jac C_{a,b,c,d,t}$ is isogenous to the 
product of $E_{a,b,c,d,t,1}$ and $E_{a,b,c,d,t,2}$.

Note that if $\lambda$ and $\mu$ are elements of $k^*$, then scaling $x$ by 
$\lambda$ and $y$ by $\mu$ in the equations for these curves and maps is 
equivalent to replacing the quintuple $(a,b,c,d,t)$ with 
$(\lambda^2 a, \lambda^3 b, \lambda^{-2} c, \lambda^{-3}d, \lambda\mu^2 t)$.
This gives an action of $k^*\times k^*$ on the set of quintuples. Note that one
can always scale a quintuple by this action in order to obtain $t = 1$.

\begin{proposition}
\label{P:3covers}
Let $k$ be a field of characteristic neither $2$ nor $3$. Suppose 
$\varphi_1\colon C\to E_1$ and $\varphi_2\colon C\to E_2$ are degree-$3$ maps 
from a genus-$2$ curve $C$ over $k$ to genus-$1$ curves $E_1$ and $E_2$ 
over~$k$, and suppose that the morphism $\varphi_{2*}\varphi_1^*$ from 
$\Jac E_1$ to $\Jac E_2$ is the zero map.  Then there are elements $a,b,c,d,t$
of $k$ satisfying~\eqref{EQ:conditions} and isomorphisms
$\alpha\colon C\to C_{a,b,c,d,t}$ and
$\alpha_i\colon E_i \to E_{a,b,c,d,t,i}$ such that the diagram
\[
\xymatrix{
E_1\ar[rr]^{\alpha_1}               && E_{a,b,c,d,t,1} \\                             
C\ar[rr]^{\alpha}\ar[d]_{\varphi_2}
                 \ar[u]^{\varphi_1} && C_{a,b,c,d,t}\ar[d]^{\varphi_{a,b,c,d,t,2}}  
                                                    \ar[u]_{\varphi_{a,b,c,d,t,1}} \\
E_2\ar[rr]^{\alpha_2}               && E_{a,b,c,d,t,2}                \\
}
\]
commutes.  The quintuple $(a,b,c,d,t)$ is unique up to the action of 
$k^*\times k^*$ given above.
\end{proposition}

The following lemma will be helpful in our proof of the proposition.

\begin{lemma}
\label{L:ramification}
Suppose $\varphi\colon C\to E$ and $\psi\colon C\to F$ are degree-$3$ maps 
from a curve $C$ to genus-$1$ curves $E$ and $F$ over a field $k$.  If 
$\varphi$ and $\psi$ have the same ramification divisor, then there is an
isomorphism $\alpha\colon E\to F$ such that $\psi = \alpha\varphi$.
\end{lemma}

\begin{remark}
The argument given by Kuhn~\cite{Kuhn}*{Corollary, p.~45} shows that $E$ and
$F$ both have rational points, so they can be made into elliptic curves.  We 
phrase the lemma and the proposition in terms of genus-$1$ curves because the
isomorphism $\alpha$ in the lemma, and the isomorphisms $\alpha_1$ and 
$\alpha_2$ in the proposition, may not be morphisms of elliptic curves --- they
do not necessarily take the identity element of one curve to the identity 
element of the other.
\end{remark}

\begin{proof}[Proof of Lemma~\textup{\ref{L:ramification}}]
Let $\omega_E$ and $\omega_F$ be nonzero holomorphic differentials on $E$ 
and~$F$. The pullbacks $\varphi^*\omega_E$ and $\psi^*\omega_F$ are holomorphic
differentials on~$C$, and the divisors of these differentials are the 
ramification divisors of the maps $\varphi$ and $\psi$. Since the ramification
divisors are equal by assumption, the two pullbacks differ by a multiplicative
constant.

Let $\kb$ be the algebraic closure of $k$, and let $C_{\kb}$ and $J_{\kb}$
be the base extensions of $C$ and its Jacobian $J$ to $\kb$.  One can embed
$C_{\kb}$ into $J_{\kb}$, and the embedding induces an isomorphism from the 
holomorphic differentials on $J_{\kb}$ to the holomorphic differentials on 
$C_{\kb}$ (see~\cite{Milne}*{Proposition~2.2}). This shows that the pullbacks 
of $\omega_E$ and $\omega_F$ to $J$ also differ by a constant, so that the 
images of $E$ and $F$ in $J$ are the same. But since the degrees of $\varphi$
and $\psi$ are prime, $E$ and $F$ are isomorphic to their images in $J$.  The
induced isomorphism $\alpha\colon E\to F$ then satisfies 
$\psi = \alpha\varphi$.
\end{proof}

\begin{proof}[Proof of Proposition~\textup{\ref{P:3covers}}]
The Riemann--Hurwitz formula shows that the map $\varphi_1$ is ramified either
at $2$ points, with ramification index $2$ (Shaska's `nondegenerate' case, and
Kuhn's `generic' case), or at one point, with ramification index~$3$ (Shaska's
`degenerate' case, and Kuhn's `special' case).  Kuhn 
shows~\cite{Kuhn}*{Lemma, p.~42} that in the former case the two ramification
points are conjugate with respect to the hyperelliptic involution, and that in
the latter case the single ramification point is a Weierstrass point.  Let
$\PP^1_C$ be the quotient of $C$ by the hyperelliptic involution. We can choose
a parameter $x$ on $\PP^1_C$ so that the $x$-coordinate of the ramification 
points (or point) is equal to~$0$.

Suppose we are in Kuhn's generic case.  Kuhn shows~\cite{Kuhn}*{\S6} that then
$C$ has a model of the form
\[
y^2 = (x^3 + \ell x^2 + m x + n)(4 n x^3 + m^2 x^2 + 2 m n x + n^2),
\]
where $n\neq 0$, and that the ramification point(s) of the map $\varphi_2$
then have $x$-coordinate equal to $-3n/m$.  

If we apply a linear fractional transformation that takes $0$ to $\infty$ and 
$-3n/m$ to~$0$, we find that the model for $C$ transforms to a curve of the 
form $C_{a,b,c,d,t}$.  Furthermore, the ramification points of the maps 
$\varphi_1$ and $\varphi_2$ are the same as the ramification points for 
$\varphi_{a,b,c,d,t,1}$ and $\varphi_{a,b,c,d,t,2}$, so by 
Lemma~\ref{L:ramification}, there are isomorphisms $\alpha$, $\alpha_1$, and 
$\alpha_2$ as in the statement of the proposition so that the diagram in the 
proposition is commutative.

Now suppose that $\varphi_1$ is degenerate, in Shaska's terminology. Arguing as
in~\cite{Shaska}*{\S2.2}, but keeping track of fields of definition, we find 
that by moving the $x$-coordinate of the ramification point(s) of $\varphi_1$
to $\infty$ and by translating and scaling $x$ appropriately, we can write $C$
as 
\[
y^2 = (3 x^2 + 4 m)(x^3 + m x + n),
\]
where $m\neq 0$; then we compute that the ramification point(s) of the map 
$\varphi_2$ have $x$-coordinate equal to $0$. Once again, by applying 
Lemma~\ref{L:ramification}, we find that there are isomorphisms $\alpha$, 
$\alpha_1$, and $\alpha_2$ as in the statement of the proposition so that the
diagram in the proposition is commutative.

To complete the proof, we must show that the quintuple $(a,b,c,d,t)$ is unique
up to the action of $k^*\times k^*$.  We obtained our model $C_{a,b,c,d,t}$ for
the curve $C$ by taking two marked points on $\PP^1_C$ --- namely, the 
$x$-coordinates of the ramification points of the maps $\varphi_1$ and
$\varphi_2$ --- and moving them to $\infty$ and $0$, respectively. That choice
determines the parameter $x$ of $\PP^1_C$ up to a scaling factor.  But the 
action of $k^*\times k^*$ on quintuples $(a,b,c,d,t)$ is exactly the action 
obtained from scaling the coordinates $x$ and $y$ for $C_{a,b,c,d,t}$.
\end{proof}

\subsection{Additional formulas}
We compute that the $j$-invariants of the elliptic curves $E_{a,b,c,d,t,1}$ 
and $E_{a,b,c,d,t,2}$ are given by
\begin{align}
\label{EQ:j1}
j(E_{a,b,c,d,t,1}) &= \frac{ 1728 (a^2 c + 4 a b d - 4 b^2 c^2)^3}
                           { \Delta_1^2 \, \Delta_2   }, \\
\label{EQ:j2}
j(E_{a,b,c,d,t,2}) &= \frac{ 1728 (a c^2 + 4 b c d - 4 a^2 d^2)^3}
                           { \Delta_1   \, \Delta_2^2 }.
\end{align}
We use these $j$-invariant formulas in Algorithm~\ref{A:glue3}.

Let $\omega_1$ and $\omega_2$ be the invariant differentials $dx/2y$ on 
$E_{a,b,c,d,t,1}$ and $E_{a,b,c,d,t,2}$, respectively. It is not hard to verify
that then we have
\[
\varphi_{a,b,c,d,t,1}^* \omega_1 = \frac{3 \, dx}{2y}
\text{\quad and\quad} 
\varphi_{a,b,c,d,t,2}^* \omega_2 = \frac{3 x\, dx}{2y}
\]
on the curve $C_{a,b,c,d,t}$.

\subsection{A note on degeneration}

Note that the map $\varphi_{a,b,c,d,t,1}$ is special (in Kuhn's terminology) 
exactly when $d = 0$, and that the map $\varphi_{a,b,c,d,t,2}$ is special 
exactly when $b = 0$. We close this appendix by explaining why our formulas
degenerate nicely to these special cases, whereas the formulas of Kuhn and 
Shaska do not.

Let $\varphi_1$ and $\varphi_2$ be as above. As Kuhn 
notes~\cite{Kuhn}*{Lemma, p.~42}, the hyperelliptic involution on $C$ descends 
via $\varphi_1$ to an involution on $E_1$ that gives a degree-$2$ map from 
$E_1$ to a projective line~$\PP^1_{E_1}$.  Then $\varphi_1$ induces a 
degree-$3$ map $\varphi_1'$ from $\PP^1_C$ to~$\PP^1_{E_1}$.

Suppose $\varphi_1$ is generic.  Then the two ramification points $P_1$ and 
$Q_1$ of $\varphi_1$ share the same image $x_1$ in $\PP^1_C$, and $x_1$ is 
doubly ramified in the triple cover $\varphi_1'$.  Let $y_1$ be the other point
of $\PP^1_C$ with $\varphi_1'(y_1) = \varphi_1'(x_1)$.

The special maps are the limiting cases that occur when $P_1$ and its involute
$Q_1$ approach a Weierstrass point of $C$.  When $P_1 = Q_1$, the point $x_1$
of $\PP^1_C$ is triply ramified in $\varphi_1'$, so the special maps can also
be viewed as the limiting cases when $y_1$ approaches $x_1$.

Both Kuhn and Shaska choose their parametrizations of generic triple covers 
$C\to E_1$ so that the points $x_1$ and $y_1$ lie at $0$ and~$\infty$.  Since
the special triple covers have $x_1 = y_1$, the parametrizations of Kuhn and 
Shaska cannot degenerate gracefully.

We have chosen our parametrization so that $x_1 = \infty$ and so that the 
corresponding point $x_2$ from the cover $\varphi_2$ lies at $0$. Since 
Lemma~\ref{L:ramification} shows that $x_1$ and $x_2$ can never be equal, there
is no reason for the parametrization to break down at the special covers.


\begin{bibdiv}
\begin{biblist}

\bib{BHP}{article}{
   author={Baker, R. C.},
   author={Harman, G.},
   author={Pintz, J.},
   title={The difference between consecutive primes. II},
   journal={Proc. London Math. Soc. (3)},
   volume={83},
   date={2001},
   number={3},
   pages={532--562},
   note={\href{http://dx.doi.org/10.1112/plms/83.3.532}
              {DOI: 10.1112/plms/83.3.532}},
}

\bib{Broker}{article}{
   author={Br{\"o}ker, Reinier},
   title={Constructing supersingular elliptic curves},
   journal={J. Comb. Number Theory},
   volume={1},
   date={2009},
   number={3},
   pages={269--273},
}

\bib{BS}{article}{
   author={Br{\"o}ker, Reinier},
   author={Stevenhagen, Peter},
   title={Efficient CM-constructions of elliptic curves over finite fields},
   journal={Math. Comp.},
   volume={76},
   date={2007},
   number={260},
   pages={2161--2179},
   note={\href{http://dx.doi.org/10.1090/S0025-5718-07-01980-1}
              {DOI: 10.1090/S0025-5718-07-01980-1}},
}

\bib{CDO}{article}{
   author={Cohen, Henri},
   author={Diaz y Diaz, Francisco},
   author={Olivier, Michel},
   title={Enumerating quartic dihedral extensions of $\mathbb{Q}$},
   journal={Compositio Math.},
   volume={133},
   date={2002},
   number={1},
   pages={65--93},
   note={\href{http://dx.doi.org/10.1023/A:1016310902973}
              {DOI: 10.1023/A:1016310902973}},
}

\bib{EL}{article}{
   author={Eisentr{\"a}ger, Kirsten},
   author={Lauter, Kristin},
   title={A CRT algorithm for constructing genus $2$ curves over finite fields},
   conference={
      title={Arithmetics, geometry and coding theory (AGCT 2005)},
   },
   book={
      editor = {F. Rodier},
      editor = {S. Vl\u adu\c t},
      series={S\'emin. Congr.},
      volume={21},
      publisher={Soc. Math. France, Paris},
   },
   date={2010},
   pages={161--176},
}

\bib{FM}{article}{
   author={Fouquet, Mireille},
   author={Morain, Fran{\c{c}}ois},
   title={Isogeny volcanoes and the SEA algorithm},
   conference={
      title={Algorithmic number theory},
      address={Sydney},
      date={2002},
   },
   book={
      editor = {C. Fieker},
      editor = {D. R. Kohel},
      series={Lecture Notes in Comput. Sci.},
      volume={2369},
      publisher={Springer},
      place={Berlin},
   },
   date={2002},
   pages={276--291},
   note={\href{http://dx.doi.org/10.1007/3-540-45455-1_23}
              {DOI: 10.1007/3-540-45455-1\_23}},
}

\bib{FK}{article}{
   author={Frey, Gerhard},
   author={Kani, Ernst},
   title={Curves of genus $2$ covering elliptic curves and an arithmetical
   application},
   conference={
      title={Arithmetic algebraic geometry},
      address={Texel},
      date={1989},
   },
   book={
      editor = {G. van der Geer},
      editor = {F. Oort},
      editor = {J. H. M. Steenbrink},
      series={Progr. Math.},
      volume={89},
      publisher={Birkh\"auser Boston},
      place={Boston, MA},
   },
   date={1991},
   pages={153--176},
}

\bib{Goursat}{article}{
   author={E. Goursat},
   title={Sur la r\'eduction des int\'egrales hyperelliptiques},
   journal={Bull. Soc. Math. France},
   volume={13},
   year={1885},
   pages={143--162},
   note={\url{http://www.numdam.org/item?id=BSMF_1885__13__143_1}},
}

\bib{Hermite}{article}{
   author={Hermite, Ch.},
   title={Sur un exemple de r\'eduction d'int\'egrales
          ab\'eliennes aux fonctions elliptiques},
   journal={Ann. Soc. Sci. Bruxelles S\'er. I},
   volume={1, 2nd part},
   year={1876},
   pages={1--16},
   note={\url{http://books.google.com/books?id=gAvjAX4hTeYC}},
}

\bib{HLP}{article}{
   author={Howe, Everett W.},
   author={Lepr{\'e}vost, Franck},
   author={Poonen, Bjorn},
   title={Large torsion subgroups of split Jacobians of curves of genus two
   or three},
   journal={Forum Math.},
   volume={12},
   date={2000},
   number={3},
   pages={315--364},
   note={\href{http://dx.doi.org/10.1515/form.2000.008}
              {DOI: 10.1515/form.2000.008}},
}

\bib{HNR}{article}{
   author={Howe, Everett W.},
   author={Nart, Enric},
   author={Ritzenthaler, Christophe},
   title={Jacobians in isogeny classes of abelian surfaces over finite fields},
   journal={Ann. Inst. Fourier (Grenoble)},
   volume={59},
   date={2009},
   number={1},
   pages={239--289},
   note={\url{http://aif.cedram.org:80/aif-bin/item?id=AIF_2009__59_1_239_0}},
}
	
\bib{Jacobi}{article}{
   author={Jacobi, C. G. J},
   title={Review of Legendre's \emph{Trait\'e des fonctions elliptiques, 
          troisi\`eme suppl\'ement}},
   journal={J. Reine Angew. Math.},
   volume={8},
   year={1832},
   pages={413--417},
   note={\url{http://resolver.sub.uni-goettingen.de/purl?PPN243919689_0008}},
}

\bib{Jacobi:Werke}{book}{
   author={Jacobi, C. G. J.},
   title={Gesammelte Werke. B\"ande I--VIII},
   series={Herausgegeben auf Veranlassung der K\"oniglich Preussischen
           Akademie der Wissenschaften. Zweite Ausgabe},
   publisher={Chelsea Publishing Co.},
   place={New York},
   date={1969},
}

\bib{Kani}{article}{
   author={Kani, Ernst},
   title={The number of curves of genus two with elliptic differentials},
   journal={J. Reine Angew. Math.},
   volume={485},
   date={1997},
   pages={93--121},
   note={\href{http://dx.doi.org/10.1515/crll.1997.485.93}
              {DOI: 10.1515/crll.1997.485.93 }},
}

\bib{Kohel}{thesis}{
   author={Kohel, David Russell},
   title={Endomorphism rings of elliptic curves over finite fields},
   type={Ph.D.~Thesis},
   organization={University of California, Berkeley},
   date={1996},
}

\bib{Kuhn}{article}{
   author={Kuhn, Robert M.},
   title={Curves of genus $2$ with split Jacobian},
   journal={Trans. Amer. Math. Soc.},
   volume={307},
   date={1988},
   number={1},
   pages={41--49},
   note={\href{http://dx.doi.org/10.2307/2000749}
              {DOI: 10.2307/2000749}},
}

\bib{Legendre}{book}{
   author={Legendre, A. M.},
   title={Trait\'e des fonctions elliptiques et des int\'egrales Eul\'eriennes, 
          Tome troisi\`eme},
   date={1828},
   publisher={Huzard--Courcier},
   place={Paris},
   note={\url{http://gallica.bnf.fr/ark:/12148/bpt6k110149h}},

}

\bib{Louboutin}{article}{
   author={Louboutin, St{\'e}phane R.},
   title={Lower bounds for relative class numbers of imaginary abelian
   number fields and CM-fields},
   journal={Acta Arith.},
   volume={121},
   date={2006},
   number={3},
   pages={199--220},
   note={\href{http://dx.doi.org/10.4064/aa121-3-1}
              {DOI: 10.4064/aa121-3-1}},
}

\bib{Matomaki}{article}{
   author={Matom{\"a}ki, Kaisa},
   title={Large differences between consecutive primes},
   journal={Q. J. Math.},
   volume={58},
   date={2007},
   number={4},
   pages={489--518},
   note={\href{http://dx.doi.org/10.1093/qmath/ham021}
              {DOI: 10.1093/qmath/ham021}},
}

\bib{Mestre}{article}{
   author={Mestre, Jean-Fran{\c{c}}ois},
   title={Construction de courbes de genre $2$ \`a partir de leurs modules},
   conference={
      title={Effective methods in algebraic geometry},
      address={Castiglioncello},
      date={1990},
   },
   book={
      editor = {T. Mora},
      editor = {C. Traverso},
      series={Progr. Math.},
      volume={94},
      publisher={Birkh\"auser Boston},
      place={Boston, MA},
   },
   date={1991},
   pages={313--334},
}

\bib{Milne}{article}{
   author={Milne, J. S.},
   title={Jacobian varieties},
   conference={
      title={Arithmetic geometry},
      address={Storrs, Conn.},
      date={1984},
   },
   book={
      publisher={Springer},
      place={New York},
   },
   date={1986},
   pages={167--212},
   note={\url{http://jmilne.org/math/articles/1986c.pdf}},
}

\bib{Shaska}{article}{
   author={Shaska, T.},
   title={Genus $2$ fields with degree $3$ elliptic subfields},
   journal={Forum Math.},
   volume={16},
   date={2004},
   number={2},
   pages={263--280},
   note={\href{http://dx.doi.org/10.1515/form.2004.013}
              {DOI: 10.1515/form.2004.013}},
}

\bib{Streng}{article}{
   author={Streng, Marco},
   title={Computing Igusa class polynomials},
   journal={Math. Comp.},
   volume={83},
   date={2014},
   number={285},
   pages={275--309},
   note={\href{http://dx.doi.org/10.1090/S0025-5718-2013-02712-3}
              {DOI: 10.1090/S0025-5718-2013-02712-3}},
}

\bib{Sutherland}{article}{
   author={Sutherland, Andrew V.},
   title={Accelerating the CM method},
   journal={LMS J. Comput. Math.},
   volume={15},
   date={2012},
   pages={172--204},
   note={\href{http://dx.doi.org/10.1112/S1461157012001015}
              {DOI: 10.1112/S1461157012001015}},
}

\bib{Tate}{article}{
   author={Tate, John},
   title={Classes d'isog\'enie des vari\'et\'es ab\'eliennes sur un corps fini 
          (d'apr\`es T. Honda)},
   pages={95--110},
   book={
     title={S\'eminaire Bourbaki. Vol. 1968/69: Expos\'es 347--363},
     series={Lecture Notes in Mathematics, Vol. 179},
     publisher={Springer-Verlag},
     place={Berlin},
     date={1971},
     pages={iv+295},
   },
   note={\href{http://dx.doi.org/10.1007/BFb0058807}
              {DOI: 10.1007/BFb0058807}},
}

\end{biblist}
\end{bibdiv}
\end{document}